\documentclass[runningheads,envcountsame]{llncs}
\usepackage[T1]{fontenc}
\usepackage{graphicx}
\usepackage{amsmath}
\usepackage{amssymb}
\usepackage{thmtools,thm-restate}
\usepackage{stmaryrd}
\usepackage{xcolor}
\usepackage{dsfont}
\usepackage[pdfusetitle,hidelinks]{hyperref}
\usepackage{enumerate}
\usepackage{wrapfig}
\usepackage{tikz}
\usepackage[capitalize]{cleveref}
\usepackage{macros}
\usetikzlibrary{cd, automata, positioning, arrows}

\begin{document}

\title{Realization of relational presheaves}


\author{Yorgo Chamoun\orcidID{0009-0007-2689-9148} \and
Samuel Mimram\orcidID{0000-0002-0767-2569}}
\authorrunning{Y. Chamoun and S. Mimram}
\institute{LIX, CNRS, École polytechnique, Institut Polytechnique de Paris, 91120 Palaiseau, France
\email{\{yorgo.chamoun,samuel.mimram\}@polytechnique.edu}}

\maketitle
\begin{abstract}
Relational presheaves generalize traditional presheaves by going to the category of sets and relations (as opposed to sets and functions) and by allowing functors which are lax. This added generality is useful because it intuitively allows one to encode situations where we have representables without boundaries or with multiple boundaries at once. In particular, the relational generalization of precubical sets has natural application to modeling concurrency.
In this article, we study categories of relational presheaves, and construct realization functors for those. We begin by observing that they form the category of set-based models of a cartesian theory, which implies in particular that they are locally finitely presentable categories. By using general results from categorical logic, we then show that the realization of such presheaves in a cocomplete category is a model of the theory in the opposite category, which allows characterizing situations in which we have a realization functor.
%
% We also study the relationship between traditional presheaves and relational ones.
% 
Finally, we explain that our work has applications in the semantics of concurrency theory. The realization namely allows one to compare syntactic constructions on relational presheaves and geometric ones. Thanks to it, we are able to provide a syntactic counterpart of the blowup operation, which was recently introduced by Haucourt on directed geometric semantics, as way of turning a directed space into a manifold.

% Relational presheaves over a category $\ccat$ are lax functors from $\ccat^{op}$ to the category $\Rel$ of sets and relations,
% and form a category $\RelPsh(\ccat)$.
% We show that this is in fact the category of set-based models of a cartesian theory,
% hence that it is in particular locally finitely presentable.
% We use general results of categorical logic to show that a realization of $\RelPsh(\ccat)$,
% i.e. a cocontinuous functor $\RelPsh(\ccat)\xrightarrow{}\dcat$ with $\dcat$ cocomplete,
% is nothing else than a model of the theory in $\dcat^{op}$.
% We also study the relationship between ordinary and relational presheaves.
% As an application, we define the geometric realization of relational precubical sets,
% which are meaningful objects in concurrency theory,
% and we show that the blowup operation commutes with this realization.

\end{abstract}

\section{Introduction}
\paragraph{Presheaves.}
The notion of \emph{presheaf} is omnipresent in modern mathematics and theoretical computer science. For instance, simplicial sets are at the heart of modern algebraic topology~\cite{may1992simplicial} and higher category theory~\cite{lurie2009higher}, and cubical sets are central in modeling truly concurrent processes through the notion of higher dimensional automaton~\cite{vanGlabbeek91,Pratt91} or achieving constructive approaches to univalent type theory~\cite{cohen2018cubical} to name a few of the myriad of occurrences of those. Here, we will be mostly interested in their ability to model transition systems, such as those arising from computational processes, concurrent ones in particular.

Formally, a presheaf is a functor $P:\ccat^\op\to\Set$, which can be understood as encoding algebraically a geometric object. Namely, an object $c$ of~$\ccat$ abstractly describes a shape, the maps of $\ccat$ describe the face operations, and the sets~$P(c)$ encode the elements of shape~$c$. Moreover, if we have a functor $\ccat\to\Top$ which describes how to associate an actual topological space to each abstract shape, we get an induced canonical functor $\Psh(\ccat)\to\Top$, called the \emph{geometric realization}, which associates a topological space to any presheaf, and this construction actually generalizes to any cocomplete category in place of~$\Top$~\cite{maclane2012sheaves}. Above, the notation $\Psh(\ccat)$ denotes the category of presheaves over~$\ccat$ and natural transformations between those. Such a category is always a Grothendieck topos (with trivial topology), and thus has very nice properties: it is always cocomplete (it is in fact the free cocompletion of~$\ccat$), complete, cartesian closed, has a subobject classifier, and so on.

As a simple example, consider the following category $\gcat$ with two objects and two non-trivial morphisms:
\vspace{-2ex}
\[
  \gcat
  \quad=\quad
  \begin{tikzcd}
    0\ar[r,shift left,"s"]\ar[r,shift right,"t"']&1
  \end{tikzcd}
\]
A presheaf $P:\gcat^\op\to\Set$ precisely corresponds to a (directed) \emph{graph} with~$P(0)$ as set of vertices, $P(1)$ as set of edges, the maps $P(s),P(t):P(1)\to P(0)$ respectively associating to each edge its source and target. Moreover, if one considers the functor $\gcat\to\Top$ sending $0$ to a point and $1$ to an interval, the image of~$s$ and~$t$ respectively being the inclusion of the point into the endpoints of the interval, the geometric realization functor $\Psh(\gcat)\to\Top$ associates to a graph the corresponding topological graph.

% Presheaves, i.e. functors 
% are omnipresent in modern mathematics and theoretical computer science.
% On the one hand, they are very expressive, 
% and they can model a huge number of widely used objects, 
% in particular those of interest for concurrency theorists
% % like directed graphs, deterministic and non-deterministic automata, simplicial and (pre)cubical sets, higher dimensional automata, etc.
% On the other hand, they are very well-behaved, 
% since the category $\Psh(\ccat)$ of presheaves over some fixed category $\ccat$ (and natural transformations between them) forms a Grothendieck topos \cite{maclane2012sheaves},
% and thus contains essentially every construction we would like to have in such a category.

\paragraph{Presheaves to relations.}
In order to be able to take into account some more situations, it is natural to generalize presheaves by replacing the category~$\Set$ by some other category~$\vcat$~\cite{kelly1982basic}. We will refrain from adopting such a general point of view here, and will be mostly interested in the case where~$\vcat=\Rel$, which allows accounting for situations where some elements of given shape have no boundary, or multiple boundaries. For instance, consider the two ``graphs'' below:
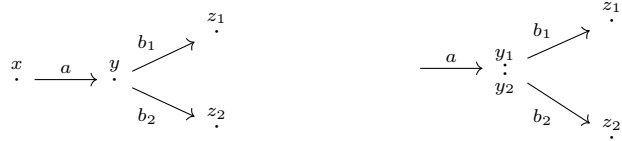
\[
  \begin{tikzcd}[row sep=0pt]
    &&\overset{z_1}\cdot\\
    \overset x\cdot\ar[r,"a"]&\overset y\cdot\ar[ur,"b_1"]\ar[dr,"b_2"']\\
    &&\overset{z_2}\cdot
  \end{tikzcd}
  \qquad\qquad\qquad
  \begin{tikzcd}[row sep=0pt]
    &&\overset{z_1}\cdot\\
    \ar[r,"a"]&\overset{y_1}{\underset{y_2}:}\ar[ur,"b_1"]\ar[dr,"b_2"']\\
    &&\overset{z_2}\cdot
  \end{tikzcd}
\]
On the left, we have drawn a graph $G:\gcat^\op\to\Set$ with $G(0)=\set{x,y,z_1,z_2}$, $G_1(1)=\set{a,b_1,b_2}$, $G(s)(a)=x$, $G(t)(a)=y$, and so on. By post-composition with the canonical functor $\Set\to\Rel$ (sending a set to itself and a function to its graph), this graph can also be seen as functor $\gcat^\op\to\Rel$. By opposition, the picture on the right represents a variant of this graph where the edge~$a$ has no vertex as source and two vertices as target. It can be represented as the presheaf~$H:\gcat^\op\to\Rel$ with
\begin{align*}
  H(0)&=\set{y_1,y_2,z_1,z_2}
  &
  H(s)(a)&=\emptyset
  &
  H(s)(b_i)&=\set{y_i}
  \\
  H(1)&=\set{a,b_1,b_2}
  &
  H(t)(a)&=\set{y_1,y_2}
  &
  H(t)(b_i)&=\set{z_i}
\end{align*}
 for $i\in\set{0,1}$ (as customary, we write a relation between $H(1)$ and $H(0)$ as a Kleisli map, \ie a function from $H(1)$ to the powerset of $H(0)$).
We can think of the edge~$a$ as representing a transition in some system which has no starting point, and makes a choice during the execution of the transition, about whether it wants to end on~$y_1$ or $y_2$: depending on this choice, only~$b_1$ or~$b_2$ can be executed afterward.
As illustrated in the previous example, going to relational presheaves allows for representing both shapes which are ``partial'' in the sense that they are lacking some boundary ($a$ has no source for instance) and ``multiple'' in the sense that they have multiple possibilities of boundaries ($a$ has both $y_1$ and $y_2$ as target).

\paragraph{Higher dimensional automata.}
In this article, we will not only be interested in presheaves over the category~$\gcat$, but rather on the \emph{cube category}~$\square$ which generalizes the previous situation, in the sense that $\gcat$ can be recovered as a full subcategory.
The objects of~$\square$ are natural numbers, where $n\in\N$ encodes the shape of an $n$-dimensional cube. Morphisms are generated by $d^\epsilon_{n,i}:n\to n+1$ with $\epsilon\in\set{-,+}$, $n,i\in\N$ with $0\leq i\leq n$, subject to the relations \[d^{\epsilon'}_{n+1,j}\circ d^\epsilon_{n,i}=d^\epsilon_{n+1,i}\circ d^{\epsilon'}_{n,j-1}\] for $0\leq i<j\leq n+1$. A morphism $d^\epsilon_{n,i}$ encodes the canonical inclusion of the $n$-cube into the $(n+1)$-cube as source or target (depending on~$\epsilon$) in direction~$i$. For instance, we have the following inclusions of the $1$-cube into the $2$-cube:
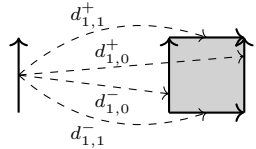
\[
  \begin{tikzpicture}
    \draw[->,thick] (-2,0) -- (-2,1);
    \fill[color=mygray] (0,0) rectangle (1,1);
    \draw[->,thick] (0,0) -- (1,0);
    \draw[->,thick] (1,0) -- (1,1);
    \draw[->,thick] (0,0) -- (0,1);
    \draw[->,thick] (0,1) -- (1,1);
    \draw (-2,.5) edge[dashed,->,"$\scriptstyle d^-_{1,0}$"',pos=.75,inner sep=0pt] (0,.25);
    \draw (-2,.5) edge[dashed,->,"$\scriptstyle d^+_{1,0}$",inner sep=0pt] (1,.75);
    \draw (-2,.5) edge[dashed,->,bend left,"$\scriptstyle d^+_{1,1}$",inner sep=0pt] (.5,1);
    \draw (-2,.5) edge[dashed,->,bend right,"$\scriptstyle d^-_{1,1}$"',inner sep=0pt] (.5,0);
  \end{tikzpicture}
\]
As expected, we can recover~$\gcat$ as the full subcategory of $\square$ on the objects~$0$ and~$1$.

\begin{wrapfigure}{r}{0.22\textwidth}
  \vspace{-1cm}
  \[
  \begin{tikzpicture}
    \def\a{{1/sqrt(2)}}
    \def\b{{2/sqrt(2)}}
    \draw (-1,0) edge[->,thick,"a"] (0,0);
    \fill[color=mygray] (0,0) -- (\a,\a) -- (\b,0) -- (\a,-\a) -- cycle;
    \draw (0,0) edge[->,thick,"$b$"] (\a,\a);
    \draw (\a,\a) edge[->,thick,"$c$"] (\b,0);
    \draw (0,0) edge[->,thick,"$c$"'] (\a,-\a);
    \draw (\a,-\a) edge[->,thick,"$b$"'] (\b,0);
  \end{tikzpicture}
\]
\vspace{-1cm}
\end{wrapfigure}
Presheaves over $\square$ are called \emph{precubical sets} and are widely used in models of true concurrency. Namely, the presence of a square can be thought of as encoding a commutation between the transitions corresponding to its edges, and $n$-cubes similarly encode commutations between $n$ transitions at once. This explains why they are at the basis of the definition of \emph{higher-dimensional automata}~\cite{vanGlabbeek91,Pratt91} (HDA) which generalize traditional automata, and constitute a very expressive model of true concurrency~\cite{vanGlabbeek06}: an HDA consists of a precubical set, with a labeling of edges ($1$-cubes), and identified initial and final states ($0$-cubes). For instance, the HDA
on the right encodes a system executing a transition~$a$, followed by two transitions~$b$ and~$c$ executed in parallel.

\paragraph{Relational presheaves.}
As indicated above, a \emph{relational presheaf} on $\ccat$ is a functor $P:\ccat^\op\to\Rel$. Our definition is actually slightly more liberal than this: it allows functors to be \emph{lax}. By this, we mean that for every composable functions~$f$ and~$g$, we have $P(g)\circ P(f)\subseteq P(g\circ f)$ and $\id\subseteq P(\id)$, \ie we allow an inclusion where an equality would have been required for traditional functors. Let us illustrate what this extra level of generality brings by illustrating it on HDA, \ie presheaves on~$\square$. As explained above, a square as $(A)$ below encodes a situation where we have two transitions~$a$ and~$b$ running in parallel:
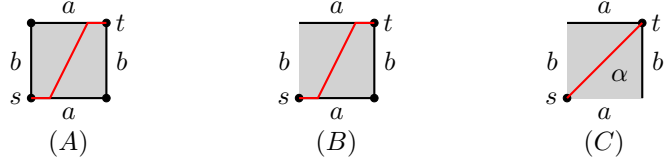
\begin{equation*}
  \begin{array}{c@{\qquad\qquad\quad}c@{\qquad\qquad\quad}c} 
    \begin{tikzpicture}[baseline=(b.base)]
      \coordinate (b) at (.5,.5);
      \fill[mygray] (0,0) rectangle (1,1);
      \filldraw (0,0) circle (0.05) node[left]{$s$};
      \filldraw (1,0) circle (0.05);
      \filldraw (0,1) circle (0.05);
      \filldraw (1,1) circle (0.05) node[right]{$t$};
      \draw
      (0,0) edge[below,thick,"$a$"'] (1,0)
      (1,0) edge[below,thick,"$b$"'] (1,1)
      (0,0) edge[below,thick,"$b$"] (0,1)
      (0,1) edge[below,thick,"$a$"] (1,1)
      ;
      \draw[thick,red] (0,0) -- (.25,0) -- (.75,1) -- (1,1);
    \end{tikzpicture}
    &
    \begin{tikzpicture}[baseline=(b.base)]
      \coordinate (b) at (.5,.5);
      \fill[mygray] (0,0) rectangle (1,1);
      \filldraw (0,0) circle (0.05) node[left]{$s$};
      \filldraw (1,0) circle (0.05);
      % \filldraw (0,1) circle (0.05);
      \filldraw (1,1) circle (0.05) node[right]{$t$};
      \draw
      (0,0) edge[below,thick,"$a$"'] (1,0)
      (1,0) edge[below,thick,"$b$"'] (1,1)
      (0,0) edge[below,draw=none,"$b$"] (0,1)
      (0,1) edge[below,thick,"$a$"] (1,1)
      ;
      \draw[thick,red] (0,0) -- (.25,0) -- (.75,1) -- (1,1);
    \end{tikzpicture}
    % &
    % \begin{tikzpicture}[baseline=(b.base)]
      % \coordinate (b) at (.5,.5);
      % \fill[mygray] (0,0) rectangle (1,1);
      % \filldraw (0,0) circle (0.05) node[left]{$\scriptstyle x_{00}$};
      % \filldraw (1,0) circle (0.05) node[right]{$\scriptstyle x_{10}$};
      % % \filldraw (0,1) circle (0.05);
      % \filldraw (1,1) circle (0.05) node[right]{$\scriptstyle x_{11}$};
      % \draw
      % (0,0) edge[below,thick,"$\scriptstyle x_{\star0}$"'] (1,0)
      % (1,0) edge[below,thick,"$\scriptstyle x_{1\star}$"'] (1,1)
      % % (0,0) edge[below,draw=none,"$b$"] (0,1)
      % (0,1) edge[below,thick,"$\scriptstyle x_{\star1}$"] (1,1)
      % ;
      % \draw (.5,.5) node {$\scriptstyle x_{\star\star}$};
      % % \draw[thick,red] (0,0) -- (.25,0) -- (.75,1) -- (1,1);
    % \end{tikzpicture}
    &
    \begin{tikzpicture}[baseline=(b.base)]
      \coordinate (b) at (.5,.5);
      \fill[mygray] (0,0) rectangle (1,1);
      \filldraw (0,0) circle (0.05) node[left]{$s$};
      % \filldraw (1,0) circle (0.05);
      % \filldraw (0,1) circle (0.05);
      \filldraw (1,1) circle (0.05) node[right]{$t$};
      \draw
      (0,0) edge[below,draw=none,thick,"$a$"'] (1,0)
      (1,0) edge[below,thick,"$b$"'] (1,1)
      (0,0) edge[below,draw=none,"$b$"] (0,1)
      (0,1) edge[below,thick,"$a$"] (1,1)
      ;
      \draw[thick,red] (0,0) -- (1,1);
      \draw (.7,.3) node{$\alpha$};
    \end{tikzpicture}
    \\
    (A)&(B)&(C)
  \end{array}
\end{equation*}
Moreover, an actual interleaving of the two transitions can be represented by a path such as the one in red in $(A)$, which corresponds to an execution where $a$ starts, then both~$a$ and~$b$ run in parallel, $b$ stops and then~$a$ stops.
%
% $a$ has to start before~$b$ (in $(C)$ is the same example, with the names of cells figured instead of labels). This can be represented by a precubical relational presheaf~$H:\square\tolax\Rel$ with
% \begin{align*}
  % H(0)&=\set{x_{00},x_{10},x_{01},x_{11}}
  % &
  % H(1)&=\set{x_{\star0},x_{1\star},x_{\star1}}
  % &
  % H(2)&=\set{x_{\star\star}}
% \end{align*}
% The face relations are such that we have $(x_{00},x_{\star0})\in H(d^-_{0,0})$ and $(x_{\star0},x_{\star\star})\in H(d^-_{1,1})$ but not $(x_{00},x_{\star\star})\in H(d^-_{0,0}\circ d^-_{1,1})$ and thus $H(d^-_{0,0}\circ d^-_{1,1})\subsetneq H(d^-_{0,0})\circ H(d^-_{1,1})$. This means that whenever we are at the starting point~$x_{00}$ (labeled $s$), we have to start executing~$x_{\star0}$ (labeled $a$) before being able to start executing~$b$ (as in the path in red): we cannot start both at the same time as in the execution corresponding to the red path in~$(D)$.
%
In $(B)$, we have represented a variant of this situation where we require that $b$ cannot start before~$a$ does, which amounts to removing the vertical boundary edge on the left.
In $(C)$, we have represented a third variant where both~$a$ and~$b$ have to start at the same time (like in the red diagonal path). This last situation can be encoded as a precubical relational presheaf~$H:\square\tolax\Rel$ with
$H(0)=\set{s,t}$, $H(1)=\set{a,b}$ and $H(2)=\set{\alpha}$.
% \begin{align*}
%   H(0)&=\set{s,t}
%   &
%   H(1)&=\set{a,b}
%   &
%   H(2)&=\set{\alpha}
% \end{align*}
The face relations are such that $(s,\alpha)\in H(d^-_{0,0}\circ d^-_{1,1})$ but $H(d^-_{1,1}(\alpha))=\emptyset$ so that $H(d^-_{0,0}\circ d^-_{1,1})\subsetneq H(d^-_{0,0})\circ H(d^-_{1,1})$, which illustrates the need for lax functors: the point $s$ is in the iterated source of the square~$\alpha$, but it is not a source of a 1-cell in the source of~$\alpha$.

\paragraph{Realization.}
The figures $(A)-(C)$ above should be handled formally: it is useful to associate, to each HDA, a topological space. More precisely, the topological space we construct should be ``directed'', \ie equipped with a notion of ``time direction'', in such a way that paths which are increasing in time correspond to actual executions (as for the red paths above), see~\cite{fajstrup2016directed} for a presentation of this approach based on directed algebraic topology. This process can be encoded by a functor called \emph{geometric realization} constructed as follows.
We have a canonical functor $M:\square\to\Top$, sending an object~$n$ to the canonical $n$-cube $[0,1]^n$. By left Kan extension along the Yoneda embedding $\yo:\square\to\Psh(\square)$, this functor induces a continuous functor $R:\Psh(\square)\to\Top$ called the \emph{realization}, which to a presheaf $P$ associates the space obtained by gluing standard cubes as described by the presheaf. This process works more generally with any cocomplete category~$\dcat$ instead of~$\Top$ (such as the category of directed topological spaces). 
%and any base category~$\ccat$ instead of~$\square$.

Here, we are interested in realizing relational presheaves, \ie defining functors $\RelPsh(\ccat)\to\dcat$ to some cocomplete category~$\dcat$. It turns out that the situation is more complicated than above: intuitively, instead of simply specifying how to realize abstract shapes, \ie objects of $\ccat$ by the functor~$M$, we also need to specify the realization of every pair of objects related by a morphism.
For example, to realize relational graphs in $\RelPsh(\gcat)$, we still have to choose a realization of $0$ and $1$, which we take to be the terminal set and the open interval $]0,1[$ respectively, but we also need to specify how to realize an egde with a source or with a target endpoint, the natural choices being~$[0,1[$ and $]0,1]$ respectively:
\[
  M(0)
  =
  \begin{tikzpicture}[baseline={([yshift=-.5ex]current bounding box.center)}]
    \filldraw (0,0) circle (.05);
  \end{tikzpicture}
  \qquad
  M(1)=
  \begin{tikzpicture}[baseline={([yshift=-.5ex]current bounding box.center)}]
    \draw[thick] (0,0) -- (.5,0);
  \end{tikzpicture}
  \qquad
  M(0\overset{s}\to 1)
  =
  \begin{tikzpicture}[baseline={([yshift=-.5ex]current bounding box.center)}]
    \filldraw (0,0) circle (.05);
    \draw[thick] (0,0) -- (.5,0);
  \end{tikzpicture}
  \qquad
  M(0\overset{t}\to 1)
  =
  \begin{tikzpicture}[baseline={([yshift=-.5ex]current bounding box.center)}]
    \filldraw (.5,0) circle (.05);
    \draw[thick] (0,0) -- (.5,0);
  \end{tikzpicture}
\]
We can then realize every relational graph as before, by gluing these building blocks together. In general, we should not expect all assignments of this kind to work (for instance, the realization of an edge and an edge with a target should be somehow related), and we provide here conditions for this. In particular, we will see that the natural first try does not satisfy those conditions and has to be modified so that it is the case.

\paragraph{Blowup.}
As an application of relational presheaves, we explain here how they can be used to provide a nice description of the operation of \emph{blowup} on geometric realizations of precubical sets introduced in~\cite{chamoun2025non,haucourt2025non}. We namely show here that this operation, which canonically turns such a space into a manifold, can be obtained as the geometric realization of a combinatorial blowup construction, turning a precubical set into a relational one. In this way, we obtain a description of the blowup operation in purely combinatorial terms, which is simpler to define and study than the topological one.

\begin{wrapfigure}{r}{0.2\textwidth}
  \vspace{-0.5cm}
  \begin{tikzpicture}
    % \draw (-1,.5) -- (1,-.5);
    % \draw (-1,-.5) -- (1,.5);
    \draw
    (-1,.5) edge["$\scriptstyle a_1$",pos=.25] (0,0)
    (-1,-.5) edge["$\scriptstyle a_2$"',pos=.25] (0,0)
    (0,0) edge["$\scriptstyle b_1$",pos=.75] (1,.5)
    (0,0) edge["$\scriptstyle b_2$"',pos=.75] (1,-.5)
    ;
    \filldraw (0,0) circle (.05) node[above]{$\scriptstyle x$};
  
    \begin{scope}[yshift=15mm]
      \draw (0,0) node{$\downarrow$};
    \end{scope}

    \begin{scope}[yshift=30mm]
      \draw
    (-1,.5) edge["$\scriptstyle a_1$",pos=.25] (0,0)
    (-1,-.5) edge["$\scriptstyle a_2$"',pos=.25] (0,0)
    (0,0) edge["$\scriptstyle b_1$",pos=.75] (1,.5)
    (0,0) edge["$\scriptstyle b_2$"',pos=.75] (1,-.5)
    ;
    \filldraw[color=white] (-.05,-.05) rectangle (.05,.05);
    % \foreach \i in {.6,.2,-.2,-.6}
    % \filldraw (0,\i) circle (.05);
    \filldraw (0,.6) circle (.05) node[above]{$\scriptstyle x_{11}$};
    \filldraw (0,.2) circle (.05) node[above]{$\scriptstyle x_{12}$};
    \filldraw (0,-.2) circle (.05) node[below]{$\scriptstyle x_{21}$};
    \filldraw (0,-.6) circle (.05) node[below]{$\scriptstyle x_{22}$};
    \end{scope}
  \end{tikzpicture}
  \vspace{-0.5cm}
\end{wrapfigure}
This construction can be illustrated on graphs. The blowup of a topological directed graph is a 1-dimensional manifold (\ie a space in which every point has a neighborhood which looks like a line), which behaves like the original graph. The way this is constructed informally consists in replacing every singular point by as many points as there are ways to traverse the singularity, equipped with a suitable topology. For instance, in the lower graph on the right the point in the middle is replaced by four points because there are four ways to go from the left to the right, thus obtaining the space above (and a projection map to the original graph).
% \[
%   \begin{tikzpicture}[baseline={([yshift=-.5ex]current bounding box.center)}]
%     % \draw (-1,.5) -- (1,-.5);
%     % \draw (-1,-.5) -- (1,.5);
%     \draw
%     (-1,.5) edge["$\scriptstyle a_1$",pos=.25] (0,0)
%     (-1,-.5) edge["$\scriptstyle a_2$"',pos=.25] (0,0)
%     (0,0) edge["$\scriptstyle b_1$",pos=.75] (1,.5)
%     (0,0) edge["$\scriptstyle b_2$"',pos=.75] (1,-.5)
%     ;
%     \filldraw (0,0) circle (.05) node[above]{$\scriptstyle x$};
%   \end{tikzpicture}
%   \qquad\qquad\rightsquigarrow\qquad\qquad
%   \begin{tikzpicture}[baseline={([yshift=-.5ex]current bounding box.center)}]
%     % \draw (-1,.5) -- (1,-.5);
%     % \draw (-1,-.5) -- (1,.5);
%     \draw
%     (-1,.5) edge["$\scriptstyle a_1$",pos=.25] (0,0)
%     (-1,-.5) edge["$\scriptstyle a_2$"',pos=.25] (0,0)
%     (0,0) edge["$\scriptstyle b_1$",pos=.75] (1,.5)
%     (0,0) edge["$\scriptstyle b_2$"',pos=.75] (1,-.5)
%     ;
%     \filldraw[color=white] (-.05,-.05) rectangle (.05,.05);
%     % \foreach \i in {.6,.2,-.2,-.6}
%     % \filldraw (0,\i) circle (.05);
%     \filldraw (0,.6) circle (.05) node[above]{$\scriptstyle x_{11}$};
%     \filldraw (0,.2) circle (.05) node[above]{$\scriptstyle x_{12}$};
%     \filldraw (0,-.2) circle (.05) node[below]{$\scriptstyle x_{21}$};
%     \filldraw (0,-.6) circle (.05) node[below]{$\scriptstyle x_{22}$};
%   \end{tikzpicture}
% \]
This construction can be described as an operation transforming a precubical set into a relational one. Indeed, the upper graph is the realization of the relational graph $H:\gcat^\op\to\Rel$ with elements and faces
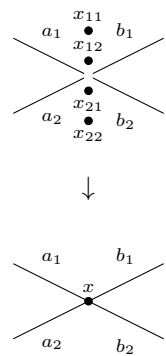
\begin{align*}
  H(0)&=\set{x_{11},x_{12},x_{21},x_{22}}
  &
  H(t)(a_1)&=\set{x_{11},x_{12}}
  &
  H(s)(b_1)&=\set{x_{11},x_{21}}
  \\
  H(1)&=\set{a_1,a_2,b_1,b_2}
  &
  H(t)(a_2)&=\set{x_{21},x_{22}}
  &
  H(s)(b_2)&=\set{x_{12},x_{22}}
\end{align*}
and no other boundaries (the point $x_{ij}$ corresponds to passing from~$a_i$ to~$b_j$).

\paragraph{Previous work.}
The idea of using functors to $\Rel$ as semantics of non-deterministic programs is quite old \cite{burstall1972algebraic}. The formal theory of relational presheaves was developed by Rosenthal in connection to nondeterministic automata theory~\cite{rosenthal1991free,rosenthal1995quantaloids} and independently by Ghilardi and Meloni in connection to modal logic~\cite{ghilardi1990modal,ghilardi1990relational}. Relational presheaves have found several applications in computer science, notably for labelled transition systems (see for example~\cite{sobocinski2015relational} for recent developments). The category of relational presheaves over $\ccat$ has been shown to be equivalent to several well-studied categories. In \cite{rosenthal1991free}, Rosenthal showed that this category is equivalent to the category of categories enriched in the free quantaloid over $\ccat$. Then, in \cite{niefield2004change,niefield2010lax}, using an analogue of the Grothendieck construction, Niefield established a correspondence between relational presheaves over $\ccat$ and faithful functors over $\ccat$, which led to a better understanding of exponentiability in the category of relational presheaves.

Higher dimensional automata (HDA) have been introduced as models for true concurrency~\cite{vanGlabbeek91,Pratt91} and bear close relationship with geometric models for concurrency~\cite{fajstrup2016directed}. The variant of \emph{partial} HDA was studied in \cite{dubut2019trees,fahrenberg2015partial} to model priorities, and constitutes a subclass of relational presheaves, corresponding to allowing partial functions for face maps.
% A partial presheaf over $\ccat$ is a lax functor from $\ccat^\op$ to the category of sets and partial functions, which is a subcategory of $\Rel$,
% so a partial presheaf is in particular a relational presheaf (and not necessarily an ordinary presheaf).
% For example, a partial graph is a relational graph where every edge has at most one source and one target.
% It was shown in \cite{fahrenberg2015partial} that, using this kind of objects, 
% it is easy to model a situation where for example two events $a$ and $b$ can occur at the same time,
% with the constraint that $a$ should \emph{start before} and \emph{finish after} $b$,
% which is a common real life situation.
%
% Directed algebraic topology
%
A way of realizing partial precubical sets (in topological spaces) was suggested in \cite{dubut2019trees}, based on the idea of completing a partial precubical set into an ordinary one.
% The idea is to `complete' a partial precubical set $P$ into an ordinary precubical set $P'$,
% which gives an monomorphism $P\hookrightarrow P'$.
% Then we can realize $P'$ in the usual way, and take the subspace corresponding to $P$ as its realization.
% This procedure is not satisfying for several reasons.
However, this procedure is not satisfying because it is not cocontinuous (consider for example the pushout square of \cite[\S 4.1]{dubut2019trees}; in fact we will see that it does not even preserve all colimits of partial presheaves which are colimits in the category of relational presheaves).

Connections between relational presheaves and categorical logic have existed since the introduction of the former.
For example, models of algebraic theories in $\Rel$ where already studied in \cite{rosenthal1995quantaloids}, and the internal logic of relational presheaves was studied in \cite{ghilardi1996relational}.
However, to the best of our knowledge, the use of categorical logic in order to study relational presheaves, as well as the definition of their realization, and the combinatorial definition of the blowup are original contributions of this paper.

\paragraph{Plan of the paper.}
Based on some classical results in categorical logic (\cref{sec:logic}), we define relational presheaves and show that they are models of a given theory (\cref{sec:main}). We then define a notion of realization, show that it can also be formulated in terms of models of a theory, and that it generalizes the usual notion of realization (\cref{sec:real}). Some concrete instances of realizations, such as in topological spaces, are then provided (\cref{sec:pcset}). We briefly introduce variants of relational presheaves and explain their use (\cref{sec:families}). Finally, we show that the blowup construction can be expressed as the realization of one on relational presheaves (\cref{sec:blowup}).

\section{Preliminaries: some categorical logic}
\label{sec:logic}
We begin by recalling some classical definitions and results on categorical logic that will be used in the following. We refer to~\cite{caramello2018theories} for details.
% A good reference for this section is \cite[\S 1.2, 1.3, 1.4, 2.1]{caramello2018theories}.
Let $\Sigma$ be a \emph{multisorted signature}, consisting of a set $\Sigma_0$ of sorts, $\Sigma_1$ of function symbols with sorted arities and a set $\Sigma_R$ of relation symbols with sorted arities.
% i.e. a set of sorts, relation symbols on finite products of sorts, and function symbols between such finite products.
Recall that a \emph{regular formula} is a first order formula built from atomic formulas using the connectives $\land$ and $\exists$.
A \emph{regular formula in context} is a finite tuple of free variables~$\mathbf{x}$ together with a regular formula $\phi$ with free variables in~$\mathbf{x}$, noted $\{\mathbf{x}.\phi\}$, and considered up to $\alpha$-conversion.
% More precisely, it is an equivalence class of such up to $\alpha$-equivalence, i.e. up to renaming the free variables.
A \emph{regular sequent} is a pair of formulas in context~$\phi$ and $\psi$, sharing the same context~$\mathbf{x}$, noted~$\phi\vdash_{\mathbf{x}}\psi$.
Its interpretation in full first order logic is $\forall\mathbf{x}.(\phi\Rightarrow\psi)$.
A \emph{regular theory} is a set of regular sequents, called the axioms of the theory.

\begin{definition}
  A \emph{cartesian formula} with respect to a regular theory $\T$ (or \emph{$\T$\nbd-carte\-sian formula}) is a regular formula such that every existential quantification
  % used in it
  is provably unique with respect to $\T$. 
  A \emph{cartesian theory} is a regular theory $\T$ such that its axioms can be well-ordered in such a way that every axiom is cartesian with respect to the theory consisting of the preceding axioms.
\end{definition}
In particular, a theory whose axioms are sequents of formulas built from atomic ones only using $\land$ is cartesian.
Now, for every cartesian (\ie finitely complete) category $\ccat$ and cartesian theory $\T$, we can define a category $\T$-$\Mod(\ccat)$ of \emph{models} of $\T$ in $\ccat$. Informally, a model consists in an interpretation of sorts as objects of~$\ccat$, of function symbols as morphisms, and of relation symbols as subobjects, in a way which respects the axioms.
Writing $\Lex$ for the category of cartesian categories and functors which are \emph{left-exact} (\ie product-preserving),
we have that $\T$-$\Mod(-)$ defines a functor on $\Lex$
and $\T$-$\Mod(\Set)$ is the usual category of set-based models of $\T$ and homomorphisms between them.
We say that a set-based model~$M$ of a cartesian theory $T$ is \emph{finitely presented} if there is a $\T$\nbd-cartesian formula in context $\{\mathbf{x}.\phi\}$ and a tuple $\mathbf{a}\in M$ such that $M\models\phi(\mathbf{a})$ and
for every model $N$ of $\T$, for every tuple $\mathbf{b}\in N$ such that $N\models \phi(\mathbf{b})$,
there is a unique homomorphism $M\to N$ sending $\mathbf{a}$ to $\mathbf{b}$.

\begin{proposition}
  \label{prop:syntactic}
  \label{syntactic}
  Let $\T$ be a cartesian theory. There is a cartesian category $\ccat_\T$, called the \emph{cartesian syntactic category} of $\T$, such that there is an equivalence
  \[
    \T\text{-}\Mod(\ccat)
    \isoto
    \Lex(\ccat_\T,\dcat)
  \]
  for every cartesian category $\dcat$, natural in $\dcat$.
  Moreover, $\ccat_\T^\op$ is equivalent to the category of finitely presented set-based models of $\T$.
\end{proposition}
For a detailed study of syntactic categories, see \cite[\S D1.4]{johnstone2002sketches}.
We mention that~$\ccat_\T$ is built in the following way:
its objects are formulas in context which are cartesian with respect to $\T$, 
and its morphisms are $\T$-provably functional cartesian formulas between them 
(more precisely, equivalence classes of such with respect to $\T$-provable equivalence). 
Cartesian theories correspond to essentially algebraic theories, for which similar theorems are established in~\cite{adamek1994locally}.
Finally, we mention that every cartesian category is the syntactic category of some cartesian theory \cite[Theorem 1.4.13]{caramello2018theories}, so that cartesian theories can actually be defined as cartesian categories.

\section{Relational presheaves as models of a cartesian theory}
\label{sec:main}
We now introduce relational presheaves and show that they form the category of models of a particular cartesian theory.
We write $\Rel$ for the category whose objects are sets and whose morphisms are relations.
We suppose fixed a small category~$\ccat$.

% % \begin{remark}\label{rem:par_rel}
% $\Par$ is a subcategory of $\Rel$, 
% since a partial function $p:A\to B$ can be seen as a relation $p\subseteq A\times B$ such that the first projection $p\to A$ is injective.
% % \end{remark}

\begin{definition}
  \label{relpsh}
  A \emph{relational presheaf} $P$ over $\ccat$ is a lax functor $\ccat^\op\to \Rel$.
  A \emph{morphism} $\alpha:P\To Q$ of relational presheaves is an oplax natural transformation.
  We write $\RelPsh(\ccat)$ for the corresponding category.
\end{definition}

\noindent
A relational presheaf~$P$ amounts to the data of a family of sets $P(c)$ indexed by objects $c\in\ccat$ and a family of relations $P(f)$ indexed by morphisms~$f$, with $P(f)\subseteq P(c)\times P(d)$ for $f:d\to c$, such that
\begin{align}
  \label{eq:relpsh-lax}
  \id_{P(c)}&\subseteq P(\id_c)
  &
  P(g)\circ P(f)&\subseteq P(f\circ g)
\end{align}
for every $c\in\ccat$ and every composable morphisms $f$ and $g$. We write $a\to _f b$ for $(a,b)\in P(f)$.
The first condition states that $P(\id_c)$ is reflexive for every $c\in\ccat$ and the second condition that if $a\to _f b$ and $b\to _g c$ then $a\to _{g\circ f} c$.
Similarly, a morphism $\alpha:P\To Q$ amounts to a family of functions $(\alpha_c:P(c)\to Q(c))_{c\in\ccat}$ such that $a\to _f b$ implies $\alpha_c(a)\to_f\alpha_d(b)$ for every $f:d\to c$.

\begin{definition}
  \label{Tcrel}
  We write $\Sigma_\ccat^\Rel$ for the signature whose sorts are the objects of~$\ccat$, with a relation symbol $R_f\subseteq c\times d$ for every morphism $f:d\to c$, and no function symbol.
  We write $\Tcrel$ for the theory with
  \begin{enumerate}
  \item for every object $c\in\ccat$, an axiom
    \[x=y  \vdash_{x:c,y:c}R_{\id_c}(x,y)\]
  \item for all morphisms $f:d\to c$, $g:e\to d$ of $\ccat$, an axiom
    \[R_f(x,y)\AND R_g(y,z) \vdash_{x:c,y:d,z:e}R_{g\circ f}(x,z) \]
  \end{enumerate}
\end{definition}

\begin{proposition}
  \label{relpsh-models}
  $\RelPsh(\ccat)$ is the category of set-based models of~$\Tcrel$.
\end{proposition}
\begin{proof}
  A model~$P$ amounts to the data of a set~$P(c)$ for each sort~$c$ and a relation~$f\subseteq P(c)\times P(d)$ for every relation $R_f$ with $f:c\to d$, and the two families of axioms ensure that the inclusions of \eqref{eq:relpsh-lax} are satisfied. This is thus precisely a relational presheaf in the sense of \cref{relpsh}. Similarly, model homomorphisms correspond to oplax natural transformations.\qed
\end{proof}
Recall that locally finitely presentable categories are exactly the categories of models of cartesian theories~\cite[Definition 9.2 and Theorem 9.8]{kelly1982structures}.
% \SM{expliquer pourquoi c'est un corollaire par une réf}

\begin{corollary}
  \label{relpsh-lfp}
  For any category $\ccat$, $\RelPsh(\ccat)$ is locally finitely presentable. Hence, it is complete and cocomplete.
\end{corollary}

In the following, we write $\tilde{\ccat}$ for the cartesian syntactic category of $\Tcrel$. By \cref{syntactic,relpsh-models}, the category of relational presheaves is precisely $\Lex(\tilde\ccat,\Set)$.
Every object $x$ of $\tilde{\ccat}$ can be seen as a relational presheaf $\yo_{\tilde{\ccat}^\op}(x)$.

% This category has all finite limits, so that its opposite has finite colimits.
% By \cref{prop:syntactic}, we deduce
% \[
  % \RelPsh(\ccat)
  % \isoto
  % \Tcrel\text{-}\Mod(\Set)
  % \isoto
  % \Lex(\tilde{\ccat},\Set)
% \]
% \SM{c'est très étrange d'avoir cela ici : c'est bien connu que les modèles de théories cartesiennes sont les lfp}
% In particular, $\RelPsh(\ccat)=\Ind((\tilde{\ccat})^\op)$ the ind-completion of $(\tilde{C})^\op$, 
% and has all small colimits by \cite[Proposition 6.1.18]{kashiwaracategories}.
% Note also that the inclusion $(\tilde{\ccat})^\op\to \Ind((\tilde{\ccat})^\op)$ is right exact by \cite[Corollary 6.1.6]{kashiwaracategories}.
%

\begin{example}
  Consider the category $\gcat$ of the introduction, whose (relational) pre\-sheaves are (relational) directed graphs. The signature $\Sigma_\gcat^\Rel$ contains two sorts~$0$ and~$1$ and two relations $R_s,R_t\subseteq 1\times 0$. The category~$\tilde\gcat$ contains the following morphisms:
  \begin{center}
    \begin{tikzcd}[row sep=small]
      & {\{(x:1,y:0).R_s(x,y)\} } \\
      {\{(y:0).\top\}} & {\{(x:1,y:0).\top\}} & {\{(x:1).\top\}} \\
      & { \{(x:1,y:0).R_t(x,y)\}}
      \arrow[hook, from=1-2, to=2-2]
      \arrow[from=2-2, to=2-1]
      \arrow[from=2-2, to=2-3]
      \arrow[hook', from=3-2, to=2-2]
    \end{tikzcd}
  \end{center}
  In a model~$G$, \ie a left exact functor $\tilde\gcat\to\Set$, the object on the left and on the right are respectively mapped to sets~$G_0$ and $G_1$, the object in the middle is mapped to the product $G_0\times G_1$, and the top and bottom object are mapped to subsets $G_s,G_t\subseteq G_0\times G_1$. We see that the object $\set{(x:1,y:0).R_s(x,y)}$ represents the ``proofs of relations'' between $0$ and $1$ via $s$. 
  % The condition of left exactness then ensures that it is mapped to a subset of the product of the images of $0$ and $1$, i.e.~that every related pair has only one proof of relation.
\end{example}

\begin{remark}
  If we apply the same procedure to the presheaf category $\Psh(\ccat)$,
  seeing it as the category of models of a theory $\T_\ccat$ similar to the relational case but with function symbols instead of relation symbols,
  the cartesian syntactic category will just be the free finite limit completion of $\ccat^\op$.
  In fact, $\tilde\ccat$ is also some kind of completion of (a modified version of) $\ccat$, see \cref{cor:completion}.
\end{remark}

\section{Realizing relational presheaves}
\label{sec:real}
A \emph{realization} of a category $\ccat$ in a cocomplete category $\dcat$ is a functor $\ccat\to\dcat$ which is cocontinuous, \ie colimit-preserving.
The fact that such a functor preserves colimits can be understood as the fact that the image of a ``complex'' element can be obtained by gluing the images of ``simpler'' elements.
We write $\coCont(\ccat,\dcat)$ for the corresponding category. Typically, for a presheaf category $\ccat=\Psh(\bcat)$, one can define a realization $\Psh(\bcat)\to\dcat$ by starting from a functor $F:\bcat\to\dcat$ and taking its left Kan extension~$F_!$ along the Yoneda embedding $\yo:\bcat\to\Psh(\bcat)$:
\[
  \begin{tikzcd}[sep=small]
    &\dcat\\
    \bcat\ar[ur,"F"]\ar[rr,"\yo"']&\ar[u,phantom,"\To"]&\Psh(\bcat)\ar[ul,dashed,"F_!"']
  \end{tikzcd}
\]
More explicitly, the realization can be computed on a presheaf $P\in\Psh(\bcat)$ as the colimit $F_!(P)=\colim_{(b,p)\in\int P}F(b)$, where $\int P$ denotes the category of elements of~$P$~\cite{mac1998categories}.
The original notion of geometric realization was essentially defined in this way with $\bcat$ being the simplicial category and $\dcat$ the category of topological spaces. Given a small category~$\ccat$ and a cocomplete category~$\dcat$, we explain how to construct a realization functor for the category $\RelPsh(\ccat)$ of relational presheaves and show that the situation is essentially analogous to the case of presheaves. We begin by observing that constructing a realization amounts to constructing a model:

\begin{proposition}
  \label{real-model}
  The category of realizations of $\RelPsh(\ccat)$ in $\dcat$ is isomorphic to the category~$\Lex(\tilde\ccat,\dcat^\op)$ of models of $\Tcrel$ in $\dcat^\op$.
\end{proposition}
\begin{proof}
  We have that relational presheaves form the Ind-completion (whose definition is recalled below) of $\tilde\ccat^\op$:
  \begin{align*}
    \RelPsh(\ccat)
    &\isoto\Tcrel\text{-}\Mod(\Set)
    &&\text{by \cref{prop:syntactic}}
    \\
    &\isoto\Lex(\tilde{\ccat},\Set)
    &&\text{by \cref{relpsh-models}}
    \\
    &\isoto\Ind(\tilde\ccat^\op)
    &&%\text{by (\ref{ind})}
  \end{align*}
  The last step can be justified as follows. Recall that the \emph{ind-completion} $\Ind(\ccat)$ of a category $\ccat$ is its free cocompletion by filtered colimits \cite{kashiwaracategories}. It can be obtained by first forming the free cocompletion $\Psh(\ccat)$ of $\ccat$ and then taking the full subcategory on presheaves which are filtered colimits of representables, or equivalently on presheaves whose category of elements is filtered, or equivalently on functors $\ccat^\op\to\Set$ which are \emph{flat} \cite[Theorem VII.6.3]{maclane2012sheaves}. It is a standard fact that if $\ccat^\op$ has all small limits, then flat functors can equivalently be defined as left-exact functors \cite[Corollary VII.6.4]{maclane2012sheaves}, \ie $\Lex(\ccat^\op,\Set)\isoto\Ind(\ccat)$. Recall that $\tilde{C}$ indeed has all finite limits, by Proposition~\ref{prop:syntactic}.

  Finally, from the above sequence of equivalences, we deduce
  \begin{align*}
    \coCont(\RelPsh(\ccat),\dcat)
    &\isoto\coCont(\Ind(\tilde\ccat^\op),\dcat)
    &&\text{by the above}
    \\
    &\isoto\Rex(\tilde{\ccat}^\op,\dcat)
    &&\text{by \cite[Proposition 1.45 (ii)]{adamek1994locally}}
    \\
    &\isoto\Lex(\tilde{\ccat},\dcat^\op)
    &&\text{by duality}
    \\
    &\isoto\Tcrel\text{-}\Mod(\dcat^\op)
    &&\text{by \cref{syntactic}}
  \end{align*}
  This concludes the proof.\qed
\end{proof}
% Now we show how to see a realization of $\RelPsh(\ccat)$ in a cocomplete category $\dcat$ as a model of $\Tcrel$ in $\dcat^\op$.
% By Proposition 1.45 (ii) of \cite{adamek1994locally}, precomposition induces an equivalence
% \[\coCont(\Ind((\tilde{\ccat})^\op),\dcat)\cong\Rex((\tilde{\ccat})^\op,\dcat)\]
% which in turn is equivalent to $\Lex(\tilde{\ccat},\dcat^\op)\cong\Tcrel\text{-}\Mod(\dcat^\op)$. 
As a consequence, in order to define a realization of $\RelPsh(\ccat)$ in $\dcat$, we need to construct a model~$M$ of $\Tcrel$ in $\dcat^\op$. This amounts to the following data:

\begin{proposition}
  \label{Tmod}
  A model in $\Mod[\Tcrel](\dcat^\op)$ is uniquely determined by the following data:
  \begin{enumerate}[A.]
  \item for every object $c\in\ccat$, an object $M(c)$ of $\dcat$,
  \item for every morphism $f:d\to  c$ of $\ccat$, an object $M(R_f)$ of $\dcat$ and morphisms $\iota_{f}^0:M(c)\to M(R_f)$ and $\iota_{f}^1:M(d)\to M(R_f)$,
  \end{enumerate} 
  such that for every $f:d\to c$ and $g:e\to d$:
  \begin{enumerate}
    \setcounter{enumi}{-1}
  \item\label{Tmod-mod} $\iota_f:=(\iota_f^0,\iota_f^1):M(c)\sqcup M(d)\to M(R_f)$ is an epimorphism,
  \item\label{Tmod-refl} the codiagonal $M(c)\sqcup M(c)\to M(c)$ factors through $\iota_{\id_c}$,
  \item\label{Tmod-trans} the canonical morphism $M(e)\sqcup M(c)\sqcup M(d)\to M(R_f)\sqcup_{M(d)}M(R_g)$ factors through $\iota_{f\circ g}\sqcup\id_{M(d)}$.%\SM{on ignore l'iso canonique sur la source, il vaudrait mieux éviter et mettre plutôt $M(e)\sqcup M(c)\sqcup M(d)\to M(R_f)\sqcup_{M(d)}M(R_g)$ ?}
  \end{enumerate}
\end{proposition}
\begin{proof}
  The data of $M(c)$ and $M(R_f$) is precisely the interpretation of the sorts and relation symbols of the signature (there is no function symbol).
  The conditions can be explained as follows, see \cite[\S 1.3.3]{caramello2018theories} for details. 
  Condition \ref{Tmod-mod} ensures that $M$ is a $\Sigma_\ccat^\Rel$\nbd-structure.
  Indeed, $M(c)$ is the object
  % (intuitively, the `set')
  of elements of sort $c$, and~$M(R_f)$ represents $\setof{(x,y)\in M(c)\times M(d)}{x\to_f y}$,
  so it should be a subobject of $M(c)\times M(d)$ in $\dcat^\op$,
  which becomes an epimorphism out of the coproduct in~$\dcat$, which by universal property of the coproduct has to be characterized by a pair of morphisms $\iota_f^0$ and $\iota_f^1$.
  % which demystifies condition 1. 
  The other two conditions ensure that~$M$ satisfies the axioms of $\Tcrel$.
  Every formula in context $\{\mathbf{x}.\phi\}$ gives rise to a subobject $M(\{\mathbf{x}.\phi\})$.
  An axiom $\phi\vdash_\mathbf{x}\psi$ is satisfied in a structure if and only if
  the subobject corresponding to $\phi$ factors through the one corresponding to $\psi$,
  and we must again write this in the opposite category,.
  This gives conditions~\ref{Tmod-refl} and~\ref{Tmod-trans} respectively associated to the two families of axioms of \cref{Tcrel} (which, in turn, correspond to the two conditions of \eqref{eq:relpsh-lax}).
  \qed
\end{proof}

We define $\ccat_\Rel$ as the full subcategory of $\tilde{\ccat}$ on the objects of the form $\{(x:c).\top\}$ and $\{(x:c,y:d).R_f(x,y)\}$ for every $f:d\to c$. This category is interesting because functors out of it encode precisely the data needed to describe a model:

\begin{lemma}
  \label{Tmod-data}
  The data of A and B in \cref{Tmod} amounts precisely to defining a functor $M:\ccat_\Rel\to\dcat^\op$.
\end{lemma}
\begin{proof}
  Since $\tilde{\ccat}^\op$ is the category of finitely presented models of $\Tcrel$,
  this category is just the category $\ccat$ where we have replaced every morphism by a span. 
  More precisely, $\Ob(\ccat_\Rel)=\Ob(\ccat)\cup\Mor(\ccat)$, 
  and there is exactly one morphism $\pi_f^1:f\to \dom(f)$ and one morphism $\pi_f^0:f\to \cod(f)$ for every $f\in\Mor(\ccat)$.
  In particular, we do not have to close by composition, because no two such morphisms are composable.\qed
\end{proof}

% Now given a model $M$, we build a functor $M:\ccat_\Rel\to\dcat^\op$ in the following way:
% \[M(\{(x:c).\top\}):=M(c),\; M(\{(x:c,y:d).R_f(x,y)\}):=M(R_f),\; M(\pi_f^i):=\iota_f^i\]
% This defines a faithful functor $F_\dcat:\Lex(\tilde{\ccat},\dcat^\op)\to\Fun(\ccat_\Rel,\dcat^\op)$ (which is in fact precomposition by the inclusion $\ccat_\Rel\hookrightarrow \tilde{\ccat}$).
By precomposition with the inclusion functor $I:\ccat_\Rel\hookrightarrow \tilde{\ccat}$, we obtain a functor $F_\dcat:\Lex(\tilde{\ccat},\dcat^\op)\to\Fun(\ccat_\Rel,\dcat^\op)$ sending a model $M$ (\ie a realization, by \cref{real-model}) to the functor $M\circ I$.
In particular, we simply write $F:\Lex(\tilde\ccat,\Set)\to\Psh(\ccat^\op_\Rel)$ for $F_\Set$.
We write $\Fun_{\Mod}(\ccat_\Rel,\dcat^\op)$ for the image of the functor $F_\dcat$: its objects consist of functors $M:\ccat_\Rel\to\dcat^\op$ (\ie interpretations of the signature by \cref{Tmod-data}) satisfying conditions 1, 2 and 3 of \cref{Tmod}.

\begin{lemma}
  \label{lem:filtered}
  The functor $F$ preserves filtered colimits.
\end{lemma}
\begin{proof}
  The functor $F$ is the composite
  $\Ind(\tilde{\ccat}^\op)\to \Psh(\tilde{\ccat}^\op)\to \Psh(\ccat_\Rel^\op)$
  where the first functor is the inclusion (see the proof of \cref{real-model}) and preserves filtered colimits by \cite[Theorem 6.1.8]{kashiwaracategories},
  and the second is precomposition and preserves small colimits since it admits a right adjoint~\cite[\S 5]{grothendieck2006sga}.
  \qed
\end{proof}

We can now give an explicit way of going from $\Lex(\tilde{\ccat},\Set)$ to the category of realizations $\coCont(\RelPsh(\ccat),\dcat)$.
In fact, we can show that this realization is essentially a left Kan extension, in the following sense. 
% Set \[F:=F_\Set:\Lex(\tilde{\ccat},\Set)\to\Psh(\ccat_\Rel^\op)\] 
Consider a model $M\in\Lex(\tilde{\ccat},\dcat^\op)$.
As we have seen in \cref{real-model}, $M$ corresponds to a realization $\hat{M}:\Lex(\tilde{\ccat},\Set)\to\dcat$.
On the other hand, by left Kan extension of~$F_\dcat(M)$ along the Yoneda embedding $\yo:\ccat_\Rel^\op\to\Psh(\ccat_\Rel^\op)$, we obtain a functor $F_\dcat(M)_!:\Psh(\ccat_\Rel^\op)\to\dcat$ which is cocontinuous.
Now $\hat{M}$ can be compared to $F_\dcat(M)_!\circ F$: we show that in fact both coincide.
\[
  \begin{tikzcd}[column sep=large]
    \Lex(\tilde\ccat,\Set)\ar[rr,bend left=15,"\hat M"]\ar[r,"F"']&\Psh(\ccat_\Rel^\op)\ar[r,dashed,"F_\dcat(M)_!"',pos=.3]&\dcat\\
    &\ccat_\Rel^\op\ar[u,"\yo"]\ar[hook,r]\ar[ur,"F_\dcat(M)"']\ar[ur,bend left=20,phantom,"\Uparrow",pos=.3]&\tilde\ccat^\op\ar[u,"M"']
  \end{tikzcd}
\]

\begin{proposition}\label{prop:kan}
  For every $M\in\Lex(\tilde{\ccat},\dcat^\op)$, we have $F_\dcat(M)_!\circ F=\hat{M}$, \ie
  \[\hat{M}(P)=\colim_{(x,a)\in \int \hspace{-0.5ex}F(P)}M(x)\]
\end{proposition}
\begin{proof}
	By~\cite[Lemma D1.4.4(ii)]{johnstone2002sketches}, any object of $\tilde{\ccat}$ is isomorphic to a formula in context which is a conjunction of atomic formulas.
	We first see that in fact any object of $\tilde{\ccat}$ is isomorphic to a conjunction of formulas of the form $R_f$, and no equalities.
  Indeed, we have
  \[\{(x_1:c,x_2:c,\mathbf{x}).(x_1=x_2)\land\phi\}\cong\{(x:c,\mathbf{x}).\phi[x_1\mapsto x,x_2 \mapsto x]\}\]
  (where $\phi[x_1=x,x_2=x]$ is the formula $\phi$ where we have substituted the variables $x_1$ and $x_2$ by the variable $x$),
  because the ``function'' from left to right sending $x_1$ and $x_2$ to $x$ is in fact functional in both directions (see \cite[Lemma D1.4.4(i)]{johnstone2002sketches}).
  Repeating this procedure a finite number of times, we get a formula with no equalities:
  we say that such a formula in context is \emph{normalized}.
	Since~$\tilde{\ccat}^\op$ is the category of finitely presented models of $\Tcrel$, 
	the morphisms between two such formulas are just maps of contexts which preserve the relations.
	First note that $F$ sends objects of $\ccat_\Rel$ to representable presheaves, since $\ccat_\Rel$ is a full subcategory of $\tilde{\ccat}$.
	Now we show that the functors coincide on $\tilde{\ccat}$, by explicit comparison. 
	This is enough to conclude, by \cref{lem:filtered} and the fact that $\Ind(\tilde{\ccat}^\op)$ is the free completion of $\tilde{\ccat}^\op$ by filtered colimits.
	Let $\{\mathbf{x}.\phi\}$ be an normalized object of $\tilde{\ccat}$. 
	Then $\{\mathbf{x}.\phi\}$ is a pullback over $\{\mathbf{x}.\top\}$ of all the $\{\mathbf{x}.R_f(x_i,x_j)\}$ such that $\phi\vdash_\mathbf{x} R_f(x_i,x_j)$ is provable in $\Tcrel$, by definition of the pullback in syntactic categories.
	Now, $\{\mathbf{x}.\top\}$ is the product of the $\{x_i.\top\}$, so we have a colimit in~$(\tilde{C})^\op$:
  \[\{\mathbf{x}.\phi\}=\colim_{(x,a)\in \int\hspace{-0.5ex}F(\{\mathbf{x}.\phi\})}x\]
  Indeed, let $D:\int F(\{\mathbf{x}.\phi\})\to(\tilde{C})^\op$ be the corresponding diagram.
  For another finitely presented model $N$, a morphism $\{\mathbf{x}.\phi\}\to N$ is equivalent 
  to a tuple $\mathbf{y}$ such that $N\models\phi(\mathbf{y})$, 
  equivalently a family of elements $(y_i)_{1\leq i\leq n}$ with $n$ the length of $\mathbf{x}$,
  such that $\phi\vdash_\mathbf{x} R_f(x_i,x_j)$ implies $N\models R_f(y_i,y_j)$,
  which in turn is exactly a morphism of diagrams from $D$ to the constant diagram at $N$.
  Now this colimit is sent by $\hat{M}$ to a colimit, so we exactly get $\hat{M}(\{\mathbf{x}.\phi\})=F_\dcat(M)_!(F(\{\mathbf{x}.\phi\}))$. 
	The proof extends to morphisms of $\tilde{\ccat}$, so we are done.\qed
\end{proof}
This is the formalization of the intuition we started with:
we realize a relational presheaf by gluing the building blocks chosen for objects and for pairs of objects related by a morphism. We can sum up the previous propositions as:

\begin{corollary}
  \label{cor:completion}
  \label{mod-to-real}
  A functor $M:\ccat_\Rel\to \dcat^\op$ with $\dcat$ cocomplete induces a cocontinuous functor $\RelPsh(\ccat)\to \dcat$ if and only if it satisfies conditions \ref{Tmod-mod}, \ref{Tmod-refl} and \ref{Tmod-trans} of \cref{Tmod}.
  % We have
  % $\Fun_{\Mod}(\ccat_\Rel,\dcat^\op)\isoto\Lex(\tilde{\ccat},\dcat^\op)$.
\end{corollary}
% In other words, $\tilde{\ccat}$ is a ``completion'' of $\ccat_\Rel$ with finite limits, in some weak sense.
Now it is easy to see that the geometric realization of relational graphs suggested in the introduction is indeed a realization.
We will give a more general and detailed construction, for precubical sets, in section~\ref{seq-real} below.
%\SM{and detailed in section... below} 

\begin{remark}
  One could also consider, instead of relational presheaves, the category $\mathrm{Pseudo}(\ccat^\op, \Span(\Set))$ of pseudofunctors from $\ccat^\op$ to the category whose objects are sets and whose morphisms are spans of functions: namely spans of sets can be seen as a ``quantitative'' variant of relations, where two elements have a set of relations between them (instead of simply being in relation or not).
  It is proven in~\cite{johnstone1999note,niefield2010lax} that, under some condition (called CFI in \cite{johnstone1999note} and IG in \cite{niefield2010lax}), 
  this is a Grothendieck topos, whose site can be obtained by putting a (non-trivial) Grothendieck topology on the twisted arrows category $\ccat_{tw}$ of $\ccat$.
  The idea is similar to what we just presented: going from $\ccat$ to $\ccat_{tw}$ essentially replaces every morphism of $\ccat$ by a ``formal'' span.
  But $\ccat_{tw}$ additionaly contains coherence data for composition of morphisms in $\ccat$.  
  The twisted arrow category approach can also be used to realize relational (more precisely, span-valued) presheaves by Kan extension.
  However, in the cases that interest us, mainly those related to precubical sets,
  the condition IG is not satisfied, 
  and in fact taking pseudofunctors instead of lax functors would significantly reduce the situations which we would be able to capture,
  as argued in the introduction.
	%\begin{figure}[h!]
	%\begin{center}
		%\begin{tikzpicture}
			%\fill[mygray] (-1,-1) rectangle (1,1) ;
			%\fill (1,1) circle (0.05);
			%\draw (0,-1.5) node{(a)};

			%\begin{scope}[xshift=40mm]
			%\fill[mygray] (-1,-1) rectangle (1,1) ;
			%\draw[thick] (-1,1) -- (1,1) ;
			%\fill (1,1) circle (0.05);
			%\draw (0,-1.5) node{(b)};
			%\end{scope}

			%\begin{scope}[xshift=80mm]
			%\fill[mygray] (-1,-1) rectangle (1,1) ;
			%\draw[thick] (-1,1) -- (1,1) ;
			%\draw[thick] (1,-1) -- (1,1);
			%\fill (1,1) circle (0.05);
			%\draw (0,-1.5) node{(c)};
			%\end{scope}
		%\end{tikzpicture}
	%\end{center}	
	%\caption{Comparing pseudo-fonctors and lax functors.}		
%\end{figure}
	%\noindent For example, the relational precubical set (a)
	%cannot be represented by a pseudofunctor, since this requires that relations compose strictly, 
	%so the relation `a vertex is at the top right corner of a square' sould be the composite of `a vertex is the source of an edge' and 'an edge is the top edge of a square', 
	%i.e.~we would need something like (b).
	%But still, this relational presheaf does not work, 
	%because by the cubical relations, 
	%the same relation should also be the composite of `a vertex is the target of an edge' and 'an edge is the right edge of a square',
	%so in fact we need (c) to be able to add a vertex to an open square, 
	%which we clearly do not want.
	
\end{remark}

We end this section by explaining the relationship between presheaves and relational ones on a fixed category~$\ccat$, and the relationship between their realizations.
There is a comparison functor $U:\Psh(\ccat)\to \RelPsh(\ccat)$, induced by post-composition with the canonical functor $\Set\to\Rel$ sending a function to the corresponding functional relation. Since both categories are locally presentable (\cref{relpsh-lfp}), we can use the special adjoint functor theorem to construct adjoints, see \cref{sec:presheaves} for a full proof.
\begin{theorem}\label{th:U}
  The comparison functor $U:\Psh(\ccat)\to \RelPsh(\ccat)$ is full and faithful and admits right and left adjoint.
\end{theorem}
We deduce that $\Psh(\ccat)$ is a full subcategory of $\RelPsh(\ccat)$ which is reflective and coreflective, hence closed under limits and colimits. In particular:
\begin{theorem}
  Every realization $\RelPsh(\ccat)\to \dcat$ induces, by precomposition with~$U$, a realization $\Psh(\ccat)\to\dcat$.
\end{theorem}
For instance, it is not difficult to see that the geometric realization of relational graphs induces the usual geometric realization of graphs.
Again, the more general case of precubical sets is treated below. 

\section{Realizations of relational precubical sets}
\label{sec:pcset}
In this section we extend some standard operations on precubical sets to relational precubical sets, and in particular define a first notion of realization.
\subsection{Barycentric subdivision}
\begin{wrapfigure}{r}{0.25\textwidth}
  \vspace{-1.1cm}
  \begin{center}
  \begin{tikzpicture}
      \fill[mygray] (0,0) rectangle (2,2);
      \filldraw (0,0) circle (0.05);
      \filldraw (2,0) circle (0.05);
      \filldraw (0,2) circle (0.05);
      \filldraw (2,2) circle (0.05);
      \draw
      (0,0) edge[below,thick,"$e$"'] (2,0)
      (2,0) edge[below,thick] (2,2)
      (0,0) edge[below,thick] (0,2)
      (0,2) edge[below,thick] (2,2)
      ;
      \draw (1,1) node{$\alpha$};
  \end{tikzpicture}
\end{center}
  \vspace{-1.5cm}
\end{wrapfigure}
As a first application of previous work, let us define the barycentric subdivision of relational precubical sets.
Let $I$ be the graph $\cdot{\to}\cdot$ with one edge, and $J$ be~$\cdot{\to}\cdot{\to}\cdot$, its barycentric subdivision (the graph with two consecutive edges).
Recall that the category of precubical sets can be equipped with a tensor product~\cite[\S 3.4.1]{fajstrup2016directed}, denoted $\otimes$.
The classical barycentric subdivision functor $\frac{1}{2}$ can be defined

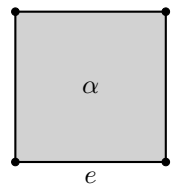
\begin{wrapfigure}{r}{0.25\textwidth}
  \vspace{-0.3cm}
  \begin{center}
  \begin{tikzpicture}
        \fill[mygray] (0,0) rectangle (2,2);
        \filldraw (0,0) circle (0.05);
        \filldraw (2,0) circle (0.05);
        \filldraw (0,2) circle (0.05);
        \filldraw (2,2) circle (0.05);
        \filldraw (1,0) circle (0.05);
        \filldraw (0,1) circle (0.05);
        \filldraw (2,1) circle (0.05);
        \filldraw (1,2) circle (0.05);
        \filldraw (1,1) circle (0.05);
        \draw
        (0,0) edge[below,thick,"\tiny ${(e,1/4)}$"'] (1,0)
        (1,0) edge[below,thick,"\tiny ${(e,3/4)}$"'] (2,0)
        (2,0) edge[below,thick] (2,2)
        (0,0) edge[below,thick] (0,2)
        (0,2) edge[below,thick] (2,2)

        (1,0) edge[below,thick] (1,1)
        (1,1) edge[below,thick] (1,2)
        (0,1) edge[below,thick] (1,1)
        (1,1) edge[below,thick] (2,1)
        ;
        \draw (1,1) node[above right]{\tiny ${(\alpha,(1/2,1,2))}$};
        \draw (0.5,0.5) node{\tiny ${(\alpha,(1/4,1,4))}$};
  \end{tikzpicture}
\end{center}
  \vspace{-1cm}
\end{wrapfigure}
\noindent as a realization functor on $\Psh(\square)$, by $n\mapsto J^{\otimes n}$.
In particular, $J^{\otimes n}$ is the barycentric subdivision of $I^{\otimes n}$.
For a precubical set $P$, we can see the elements of $\frac{1}{2}P$ as pairs $(x,(t_1,\dots,t_m))$ with $x\in P(m)$ and $t_j\in\{1/4,1/2,3/4\}$.
The dimension of $(x,(t_1,\dots,t_m))$ is given by the number of indices $j$ such that $t_j\ne1/2$.
See \cite[\S 5.1]{chamoun2025non} for details.
We say that $(x,(t_1,\dots,t_m))$ \emph{belongs to} $x$.

% As an application of the tools we developped, let us define a subdivision realization functor for relational precubical sets.
By~\cref{mod-to-real}, defining a subdivision realization functor amounts to defining a model~$M$ of $\Tcrel[\square]$ in $\RelPsh(\square)$, which we now do.
To an object $n\in\square$, we associate the relational precubical subset of $J^{\otimes n}$ consisting of the cubes belonging to the unique $n$-cube of $I^{\otimes n}$,
so that this is the subdivision of the open $n$-cube $\yo_{\tilde{\ccat}^\op}(\{(x:n).\top\})$.
We also set $M(R_{\id_n}):=M(n)$.
Any other morphism $f:n\to n+k$ in~$\square$ can be seen as an inclusion of an $n$-cube into an $(n+k)$-cube, and we define $M(R_f)$ to be the subdivision of the open $(n+k)$-cube together with the corresponding subdivided open $n$-cube in its boundary.
More precisely, $\yo_{\tilde{\ccat}^\op}(\{(x:n+k,y:n).R_f(x,y)\})$ is canonically a subobject of the representable $(n+k)$-cube $I^{\otimes n}$ seen as a relational precubical set.
We then define $M(R_f)$ to be the relational precubical subset of $\frac{1}{2}I^{\otimes n}$ consisting of the cubes wich belong to $\yo_{\tilde{\ccat}^\op}(\{(x:c,y:d).R_f(x,y)\})$.
For example:
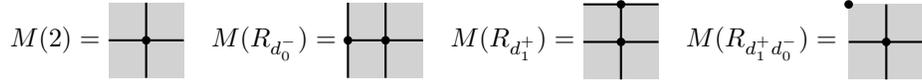
\begin{align*}
  M(2)&=
  \begin{tikzpicture}[baseline={([yshift=-.5ex]current bounding box.center)}]
    \fill[color=mygray] (0,0) rectangle (1,1);
    \filldraw (0.5,0.5) circle (0.05);
      \draw
      (0,0.5) edge[below,thick] (1,0.5)
      (0.5,0) edge[below,thick] (0.5,1)
      ;
  \end{tikzpicture}
  &
  M(R_{d_0^-})
  &=
  \begin{tikzpicture}[baseline={([yshift=-.5ex]current bounding box.center)}]
    \fill[color=mygray] (0,0) rectangle (1,1);
    \draw[thick] (0,0) -- (0,1);
    \filldraw (0.5,0.5) circle (0.05);
    \filldraw (0,0.5) circle (0.05);
      \draw
      (0,0.5) edge[below,thick] (1,0.5)
      (0.5,0) edge[below,thick] (0.5,1)
      ;
  \end{tikzpicture}
  &
  M(R_{d_1^+})
  &=
  \begin{tikzpicture}[baseline={([yshift=-.5ex]current bounding box.center)}]
    \fill[color=mygray] (0,0) rectangle (1,1);
    \draw[thick] (0,1) -- (1,1);
    \filldraw (0.5,0.5) circle (0.05);
    \filldraw (0.5,1) circle (0.05);
      \draw
      (0,0.5) edge[below,thick] (1,0.5)
      (0.5,0) edge[below,thick] (0.5,1)
      ;
  \end{tikzpicture}
  &
  M(R_{d_1^+d_0^-})
  &=
  \begin{tikzpicture}[baseline={([yshift=-.5ex]current bounding box.center)}]
    \fill[color=mygray] (0,0) rectangle (1,1);
    % \draw[thick] (0,0) -- (0,1);
    \filldraw (0,1) circle (.05);
    \filldraw (0.5,0.5) circle (0.05);
      \draw
      (0,0.5) edge[below,thick] (1,0.5)
      (0.5,0) edge[below,thick] (0.5,1)
      ;
  \end{tikzpicture}
\end{align*}
We can check without difficulty that this indeed defines a model, thus a realization.
In addition, the induced realization on $\Psh(\square)$ is exactly the usual barycentric subdivision.

\subsection{Sequential geometric realization}
\label{seq-real}
Now, let us define a geometric realization functor for relational precubical sets: we will see that the situation there is quite subtle.
Again, we do it by defining a model~$M$ of $\Tcrel[\square]$ in $\Top^\op$.
Following the idea of \cite{dubut2019trees} for realizing partial precubical sets, 
it is natural to start by trying the following construction.
To an object $n\in\square$, we associate the topological open $n$-cube $M(n):=\mathopen]0,1[^n$.
A morphism $f:n\to n+k$ in~$\square$ can be seen as an inclusion of an $n$-cube into an $(n+k)$-cube and $M(R_f)$ is defined as an $(n+k)$-cube together with the corresponding $n$-cube in its boundary.
More precisely, such a morphism decomposes uniquely as
\[
  f=d^{\epsilon_k}_{n+k-1,i_k}\circ\ldots \circ d^{\epsilon_2}_{n+1,i_2}\circ d^{\epsilon_1}_{n,i_1}
\]
with $0\leq i_1<i_2<\ldots<i_k<n+k$, see~\cite{grandis2003cubical}. We define
% $M(R_f)=\prod_{0\leq i<n+k}I_i$ where $I_i=[0,1[$ (\resp $I_i=]0,1]$) if $i=i_k$ for some~$k$ and $\epsilon_k=-$ (\resp $\epsilon_k=+$), and $I_i=[0,1]$ otherwise.
$M(R_f)$ to be the subspace of $[0,1]^n$ consisting of points $(x_1,\ldots,x_n)$ such that either all the $x_i$ belong to~$]0,1[$, or for every index $i$ we have $x_i=0$ (\resp $x_i=1$) if $i=i_j$ for some $j$ with $\epsilon_j=-$ (\resp $\epsilon_j=+$) and $x_i\in\mathopen]0,1[$ otherwise.
In particular, $M(R_{\id_n})=M(n)$.
For instance,
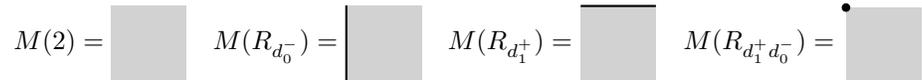
\begin{align*}
  M(2)&=
  \begin{tikzpicture}[baseline={([yshift=-.5ex]current bounding box.center)}]
    \fill[color=mygray] (0,0) rectangle (1,1);
  \end{tikzpicture}
  &
  M(R_{d_0^-})
  &=
  \begin{tikzpicture}[baseline={([yshift=-.5ex]current bounding box.center)}]
    \fill[color=mygray] (0,0) rectangle (1,1);
    \draw[thick] (0,0) -- (0,1);
  \end{tikzpicture}
  &
  M(R_{d_1^+})
  &=
  \begin{tikzpicture}[baseline={([yshift=-.5ex]current bounding box.center)}]
    \fill[color=mygray] (0,0) rectangle (1,1);
    \draw[thick] (0,1) -- (1,1);
  \end{tikzpicture}
  &
  M(R_{d_1^+d_0^-})
  &=
  \begin{tikzpicture}[baseline={([yshift=-.5ex]current bounding box.center)}]
    \fill[color=mygray] (0,0) rectangle (1,1);
    % \draw[thick] (0,0) -- (0,1);
    \filldraw (0,1) circle (.05);
  \end{tikzpicture}
\end{align*}
% \SM{ça ne marche pas !}
% We also set $R_{\id_n}=]0,1[^n$ for every $n$.
% Now a morphism $w:m\to n$ in $\square$ represents the face inclusion of an $m$-cube into an $n$-cube, 
% which can in turn be seen as an element of $\{-1,0,1\}^n$ with exactly $m$ non-zero components.
% Writing $I_{-1}:=[0,1[, \, I_0:=]0,1[, \, I_1:=]0,1]$, we define $M(R_w):=\Pi_{i\in w}I_i$, with the obvious inclusions.
%
However, this construction does not work, because $M$ is not a model of $\Tcrel[\square]$. For instance, the condition \ref{Tmod-trans} of \cref{Tmod} is not satisfied.
Indeed, consider the obvious inclusion $f:M(R_{d_1^-d_0^-})\hookrightarrow M(R_{d_0^-})\sqcup_{M(1)}M(R_{d_0^-})$:
\[
  f
  \quad:\quad
  \begin{tikzpicture}[baseline={([yshift=-.5ex]current bounding box.center)}]
    \fill[color=mygray] (0,0) rectangle (1,1);
    \filldraw (0,0) circle (.05);
    \draw[thick,red] (0,0) -- (.5,.5);
  \end{tikzpicture}
  \quad\into\quad
  \begin{tikzpicture}[baseline={([yshift=-.5ex]current bounding box.center)}]
    \fill[color=mygray] (0,0) rectangle (1,1);
    \draw[thick] (0,0) -- (0,1);
    \filldraw (0,0) circle (.05);
    % \draw (0,0) node[left]{$s$};
  \end{tikzpicture}
\]
This map is not continuous, since the space in the target has more open sets than the one in the source.
For instance, on the space in the source we have a diagonal path (drawn in red) starting from the lower-left corner and going inside the square, but its image in the space in the target is not continuous.
In fact, the image of~$f$ is topologically equivalent to the space $M(0)\sqcup M(1)$, see \cref{cor:ht}.
% For instance, there is no continuous path in the space on the right, whose starting point is~$s$ and 
% there are no continuous paths $\gamma:[0,1]\xrightarrow{}M(R_{d_0^-})\sqcup_{M(1)}M(R_{d_0^-})$ such that $\gamma(0)=s$ and $\gamma(]0,1])$ is in the interior of the square.
% In fact, the induced topology on the subset of $M(R_{d_0^-})\sqcup_{M(1)}M(R_{d_0^-})$ represented by $M(R_{d_1^-d_0^-})$ has the same homotopy type as the coproduct space $M(0)\sqcup M(1)$ (see Corollary \ref{cor:ht}).
% This violates condition \ref{Tmod-trans} of \cref{Tmod}.
There are essentially two solutions to this problem, we provide a first one here, and another one in next section.

\bigskip

% \paragraph{The sequential geometric realization.}
A first solution is to slightly modify the above definition in the following way.
We consider a model~$M$ of $\Tcrel[\square]$ in $\Top^\op$, which is defined as above on objects and on morphisms of the form $R_{\id_n}$ and $R_{d_{n,i}^\epsilon}$. However, for any other morphism $f:m\to n$, we define $M(R_f):=M(m)\sqcup M(n)$, as suggested by the above observation.
This indeed satisfies condition \ref{Tmod-trans} of \cref{Tmod}, because the topology on $M(R_{f\circ g})$ for $f,\,g\ne\id$ will now always be the coproduct topology,
\ie the topology with the most possible open sets, so that no continuity problems can occur.
This solution is interesting, because we get a space where continuous paths can only go from (the realization of) a cube to an adjacent cube of dimension $+1$ or $-1$.
This is coherent with the combinatorial definition of paths of partial HDAs in \cite{dubut2019trees,fahrenberg2015partial},
where one is only allowed to go from a cube to an adjacent cube of dimension $+1$ or $-1$,
which in turn is crucial for the interpretation of partial HDAs in formalizing priorities.
Indeed, going directly from a vertex to a square means starting two transitions at the same time,
so it is impossible to force one of them to start \emph{strictly} before the other if this is allowed.
Since in a path, the creation of multiple processes has do be done in a sequential way, we introduce the following terminology:

\begin{definition}
  The \emph{sequential geometric realization} is the functor induced by the above model.
   % $\real{{-}}^{\textnormal{seq}}_\Top:\preRelPsh(\square)\to\Top$
\end{definition}

\section{Variants of relational presheaves}
\label{sec:families}
Many variants of the notion of precubical relation can be thought of. For instance, one could take colax functors instead of lax ones, \ie we reverse the inclusions of \eqref{eq:relpsh-lax} in the definition. In this section, we will be mostly interested in another variant that we call \emph{relational families},
which is the variant of \cref{relpsh} where we simply drop both conditions \eqref{eq:relpsh-lax}: a relational family on a category~$\ccat$ consists of a family of sets indexed by objects of~$\ccat$ and a family of relations indexed by morphisms of~$\ccat$.
We still keep oplax natural transformations as morphisms and write $\RelFam(\ccat)$
for the resulting category.
All the results of the paper hold for the variants of relational presheaves if we modify appropriately the conditions, see \cref{sec:variants} for a detailed study. In particular, for relational families, one should remove both families of axioms in \cref{Tcrel} and both conditions \ref{Tmod-refl} and \ref{Tmod-trans} in \cref{Tmod} in order to have the above theorems holding for relational families instead of relational presheaves.

% There is another solution, which gives rise to a realization with a more natural topology: we can simply drop condition \ref{Tmod-trans} of \cref{Tmod}.
% Indeed, instead of looking at relational precubical sets, or more generally relational presheaves,
% we could have written essentially the same paper without requiring the second condition of \eqref{eq:relpsh-lax}.
% For a fixed base category $\ccat$, we call such objects \emph{$\Rel$-families}, and we note their category by $\preRelPsh(\ccat)$.
% They essentially behave like relational presheaves, 
% and their realizations are defined in the same way,
% with the exception that the corresponding functor $\ccat_\Rel\to\dcat^\op$ does not have to satisfy condition \ref{Tmod-trans} of \cref{Tmod} (see Appendix \ref{sec:variants} for a detailed study).
\begin{wrapfigure}{r}{0.15\textwidth}
  \vspace{-12mm}
  \begin{center}
    \begin{tikzpicture}
      \fill[mygray] (0,0) rectangle (1,1);
      \filldraw (0,0) circle (0.05) node[left]{$s$};
      \filldraw (1,0) circle (0.05);
      % \filldraw (0,1) circle (0.05);
      \filldraw (1,1) circle (0.05) node[right]{$t$};
      \draw
      (0,0) edge[below,thick,"$a$"'] (1,0)
      (1,0) edge[below,thick,"$b$"'] (1,1)
      (0,0) edge[below,draw=none,"$b$"] (0,1)
      (0,1) edge[below,thick,"$a$"] (1,1)
      ;
    \end{tikzpicture}
  \end{center}
  \vspace{-1cm}
\end{wrapfigure}
The category of relational families is interesting for the semantics of concurrency.
For example, consider the square on the right. 
We can now allow a path in an HDA to jump from a cube to some adjacent cube of arbitrary dimension, which is more standard.
Namely, since there are no conditions on composites, we can choose whether $s$ should be or not the $0$-face of the square, which amounts to determining whether the transitions $a$ and $b$ can or cannot start exactly at the same time.
The $\Sigma_\square^\Rel$-structure $M$ defined at the beginning of \cref{seq-real} (before modifications) is a model in this new setting, and therefore defines a realization $\RelFam(\square)\to\Top$.

\begin{definition}
  The \emph{geometric realization} $\real{{-}}_\Top:\RelFam(\square)\to\Top$ is the functor induced by the above model.
\end{definition}
We use the same notation and terminology for the functor $\RelPsh(\square)\to\Top$ obtained by precomposing $\real{{-}}_\Top$ with $\RelPsh(\square)\hookrightarrow\RelFam(\square)$ the canonical inclusion.
By \cref{rem:psh-to-pre}, this realization induces a (cocontinuous) realization of precubical sets $\Psh(\square)\to\Top$. Unlike for the sequential geometric realization, we get:

\begin{restatable}{proposition}{inducedrestate}
  \label{prop:induced}
  $\real{{-}}_\Top\circ U$ is the geometric realization of precubical sets. 
\end{restatable}

\section{Blowup commutes with realization}\label{sec:blowup}
Suppose fixed~$n\in\N$. The blowup of a (locally ordered) space is a best approximation of it by an $n$-euclidean space:

\vspace{-6ex}
~
\begin{wrapfigure}{r}{0.2\textwidth}
  \vspace{-5mm}
  \[
  \begin{tikzcd}
    & \Tilde{X}\arrow[d,"\beta_X"]\\
    E \arrow[r,"f"'] \arrow[ur, "\Tilde{f}", dashed] & X
  \end{tikzcd}
\]
\vspace{-1.5cm}
\end{wrapfigure}

\begin{definition}
  Given a locally ordered space~$X$, a \emph{blowup} is an $n$-euclidean locally ordered space $\tilde{X}$ equipped with a local embedding $\beta_X:\tilde{X}\to X$ such that for every $n$-euclidean local order $E$ and every local embedding $f:E\to X$, there is a unique continuous lift $\tilde{f}$ such that $\beta_X\circ\tilde f=f$:
\end{definition}
The blowup of a locally ordered space always exists and is unique up to isomorphism~\cite{chamoun2025non}.
%
% , fixing $n\in\N$, for a locally ordered space $X$, 
% there is a unique (up to isomorphism) locally ordered space $\tilde{X}$, called the \emph{blowup} of~$X$, equipped with a morphism $\beta_X:\tilde{X}\to X$ such that for every $n$-euclidean local order $E$ and every local embedding $f:E\to X$, there is a unique continuous lift $\tilde{f}$ making the following diagram commute:
% \begin{center}
  % \begin{tikzcd}
    % & \Tilde{X}\arrow[d,"\beta_X"]\\
    % E \arrow[r,"f"'] \arrow[ur, "\Tilde{f}", dashed] & X
  % \end{tikzcd}
  % \end{center}
% $\tilde{X}$ is called the blowup of $X$, and $\beta_X$ is called the blowup map.
%Moreover, $\tilde{f}$ is a local embedding.
A combinatorial description of the blowup is given in \cite[section 5]{chamoun2025non}, in the case where $X$ is the realization of a precubical set~$P$.
In this section, we show how this description can be expressed as a relational precubical set~$\tilde{P}$ over~$P$, and then we prove that the realization of $\tilde{P}$ as defined here is exactly the underlying topological space of~$\tilde{X}$,
see \cref{sec:comb-bup} for proofs and details.

\begin{definition}
  Let $P$ be a relational precubical set, $c$ a cube of $P$.
  The \emph{neighborhood} $N(c)$ of $c$ is the relational precubical subset of $P$ consisting of the cubes~$c'$ such that there exists $f$ with $c'\to_f c$.
\end{definition}

\begin{wrapfigure}{r}{0.15\textwidth}
  \vspace{-0.3cm}
  \begin{center}
    \begin{tikzpicture}[baseline={([yshift=-.5ex]current bounding box.center)}]
    \fill[color=mygray] (0,0) rectangle (1,1);
    \draw[thick] (0.5,0) -- (0.5,1);
    \draw[thick] (0,0) -- (0,1) -- (1,1) -- (1,0) -- (0,0);
    \filldraw (0.5,0) circle (.05);
    \filldraw (0.5,1) circle (.05);
    \filldraw (0,0) circle (.05);
    \filldraw (1,1) circle (.05);
    \filldraw (0,1) circle (.05);
    \filldraw (1,0) circle (.05);
  \end{tikzpicture}
  \end{center}
    \vspace{-1cm}
\end{wrapfigure}
\noindent
Recall that $I$ is the graph $\cdot{\to}\cdot$ with one edge, and $J$ is $\cdot{\to}\cdot{\to}\cdot$ the graph with two consecutive edges, and write $I_{k,m}:=I^{\otimes k}\otimes J^{\otimes m}$.
% has a unique cube of minimal dimension $k$.
Recall that a \emph{symmetric} precubical set is a precubical set $P$ such that each $P(k)$ is endowed with an action of the symmetric group $\mathfrak{S}_k$,
such that the face maps are compatible with these actions, see \cite[\S 6]{grandis2003cubical}.
Every precubical set can canonically be seen as a symmetric precubical set, by freely adding the images of these actions.
Intuitively, a $k$-cube of a symmetric precubical set $P$ is represented by the orbit of an element of $P(k)$, formalizing the idea that the order on the dimensions does not matter.
For example, $I_{1,1}$ is not isomorphic to the precubical set in the upper right corner. However, they are isomorphic as symmetric precubical sets.
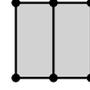
\begin{definition}
  A \emph{$(n,k)$-euclidean brick} is the neighborhood of the minimal cube of some precubical set which is isomorphic to $I_{k,n-k}$ as \emph{symmetric} precubical sets.
\end{definition}
For example:
\vspace{-1.5ex}
\begin{align*}
  N(\min(I_{2,0}))
  &=
  \begin{tikzpicture}[baseline={([yshift=-.5ex]current bounding box.center)}]
    \fill[color=mygray] (0,0) rectangle (1,1);
  \end{tikzpicture}
  &
  N(\min(I_{1,1}))
  &=
  \begin{tikzpicture}[baseline={([yshift=-.5ex]current bounding box.center)}]
    \fill[color=mygray] (0,0) rectangle (1,1);
    \draw[thick] (0,0.5) -- (1,0.5);
  \end{tikzpicture}
  &
  N(\min(I_{0,2}))
  &=
  \begin{tikzpicture}[baseline={([yshift=-.5ex]current bounding box.center)}]
    \fill[color=mygray] (0,0) rectangle (1,1);
    \draw[thick] (0.5,0) -- (0.5,1);
    \draw[thick] (0,0.5) -- (1,0.5);
    \filldraw (0.5,0.5) circle (.05);
  \end{tikzpicture}
\end{align*}
\vspace{-2ex}
\begin{definition}
  A morphism of relational precubical sets $\alpha:P\to Q$ is a \emph{local embedding} if $a\to_f c$ and $b\to_f c$ implies $\alpha(a)\ne\alpha(b)$ for cubes $a,\, b$ and $c$ of $P$.
\end{definition}
Euclidean bricks are typically precubical sets whose geometric realization is isomorphic to $\R^n$, 
and the realization of a local embedding is a local embedding.
Note that $I_{k,m}$ has a unique cube of dimension $k$, and so does any euclidean brick $B$: we write $\min(B)$ for this cube, since all other cubes of $B$ have dimension $>k$.

\begin{definition}\label{def:bup}
  Let $P$ be a precubical set of maximum dimension $n$. 
  The \emph{blowup} of $P$ is the relational precubical set $\tilde{P}$ given by
  \begin{align*}
    \tilde{P}(k) = \{ S\subseteq P\,|\, & \exists \alpha : B \to S \text{ surjective local embedding,}\\
                                        & \text{for some $(n,k)$-euclidean brick } B\}\\
    \tilde{P}(f) = \{ (S,S')\,|\,& \exists \alpha : B \to S,\,\exists \alpha' : B' \to S' \text{as above, }\exists \iota:B\to B' \text{ monic}\\
                                 & \text{such that }\alpha(\min(B))\to_f \alpha'(\min(B'))  \text{ and } \alpha'\circ\iota=\alpha\}
  \end{align*}
  where `$\subseteq$' means `subobject of'. The \emph{blowup map} $\beta_P:\tilde{P}\to P$ is defined by $\beta_P(S):=\alpha(\min(B))$ for $\alpha:B\to S$ as above. 
\end{definition}
This corresponds to the combinatorial blowup of~\cite{chamoun2025non}.
More precisely, the combinatorial blowup of a precubical set $P$ is expressed there as a pointwise subset of a presheaf $\mathsf{Comb}_P:(\int P)^\op\to\Set$ on the opposite of the category of elements $\int P$ of $P$. And this corresponds to a relational presheaf, because we have an equivalence between this category of presheaves and a subcategory of relational presheaves over~$P$, which generalizes the classical equivalence between presheaves and discrete fibrations~\cite[Theorem 2.1.2]{loregian2020categorical}. Namely, define the category~$\DFib(P)$ of \emph{discrete fibrations} over~$P$ as the full subcategory of the slice category $\RelPsh(\ccat)/P$ whose objects have the unique right lifting property with respect to the inclusions $\yo_{\tilde{\ccat}^\op}(\{(y:d).\top\})\hookrightarrow \yo_{\tilde{\ccat}^\op}(\{(x:c,y:d).R_f(x,y)\})$. We then have,

\begin{restatable}{theorem}{dfibrestate}
  \label{thm:psh_to_fib}
  Let $P$ be a relational presheaf over a category $\ccat$.  There is an equivalence of categories $\phi:\Psh((\int P)^\op)\cong\DFib(P)$.
\end{restatable}

\noindent
Transforming $\mathsf{Comb}_P$ into a relational presheaf over $P$ by~$\phi$ and taking the relational precubical subset corresponding to the combinatorial blowup, we get~$\tilde{P}$.

Now, given $F\in \Psh((\int P)^\op)$, there is a procedure defined in~\cite{chamoun2025non} which can be described as a geometric realization.
First, we transform $F$ into a sheaf on $|P|_\Top$, using the fact that this space has a basis consisting of open sets ``centered'' at a cube of minimal dimension.
Then, we use the correspondence between sheaves on a space and local homeomorphisms over this space~\cite[\S II.6]{maclane2012sheaves} to get a space over $|P|_\Top$.
Writing $|F|$ for this space, we can prove the following.
\begin{restatable}{theorem}{rblowuprestate}
  \label{th:real-blowup}\label{th:real-blowup}
  Let $P$ be a relational precubical set, $F\in \Psh((\int P)^\op)$.
  Then we have $|F|\cong|\phi(F)|_\Top$ as spaces over $|P|_\Top$.
\end{restatable}

\noindent The idea is to first establish a set-theoretic bijection over $|P|_\Top$. Since both bundles are local homeomorphisms, we directly get that this bijection is in fact a homeomorphism.
Finally, we can prove the desired theorem.
\begin{theorem}
  Let $P$ be a precubical set of maximum dimension $n$.
  The geometric realization of the blowup $\tilde{P}$ of $P$ is the underlying space of the blowup~$\tilde{X}$ of the locally ordered realization $X$ of $P$.
\end{theorem}
\begin{proof}
  First, we need to complete $\tilde{P}$ in order to get a discrete fibration $\tilde{P}^+$ over~$P$.
  We set $\tilde{P}^+(k):=\tilde{P}(k)\sqcup P(k)$ (the elements coming from $P(k)$ correspond to the elements $\bot$ of $\mathsf{Comb}_P$), and $(a,b)\in \tilde{P}^+(f)$ if and only if $(a,b)\in\tilde{P}(f)$ or $(a,\beta_P(b))\in P(f)$ and there are no $a'$ such that $(a',b)\in\tilde{P}(f)$.
  Now, by \cref{thm:psh_to_fib}, we can transform $\tilde{P}^+$ into a presheaf $\phi^{-1}(\tilde{P}^+)$ on $(\int P)^\op$.
  Simple but tedious calculations (see \cref{lem1,lem2} and the discussion before them) give $\phi^{-1}(\tilde{P}^+)=\mathsf{Comb}_P$.
  In~\cite[Theorem 5.17]{chamoun2025non}, it is proven that the underlying space of the blowup $\tilde{X}$ is a well-characterized subspace of $|\phi^{-1}(\tilde{P}^+)|$, corresponding to the elements coming from $\tilde{P}$. 
  \cref{th:real-blowup} concludes.
  \qed
\end{proof}

% Now we are ready to prove the theorem.

% \begin{theorem}
% 	Let $P$ be a precubical set of maximum dimension $n\in\N$. 
% 	The geometric realization of the $n$-blowup $\tilde{P}$ of $P$ is the underlying space of the blowup $\tilde{X}$ of the locally ordered realization $X$ of $P$.
% \end{theorem}
% \begin{proof}
% 	By e.g. Theorem 5.12 of \cite{chamoun2025non}, the fiber of $|\tilde{P}|_\Top\to |P|_\Top$ over $x\in\{e\}\times]0,1[^{\dim(e)}$ is exactly the set of germs of euclidean embeddings of $X$ at $x$.
% 	So we have a canonical fiberwise bijection between $|\tilde{P}|$ and $\tilde{X}$. 
% 	We just need to prove that this is open and continuous.
% 	But this is obvious by the concrete description of the topology of $|\tilde{P}|_\Top$ given in lemma \ref{lem:basis},
% 	and the interpretation of the relations of $\tilde{P}$, where $S\to _f S'$ exactly means that, 
% 	seeing $S$ and $S'$ as germs of embeddings at points of cubes $e$ and $e'$ respectively, with $e\to _f e'$, 
% 	$S'$ restricts to $S$ on a neighborhood of a point of $e$ which is contained in a neighborhood of a point of $e'$.
% 	\qed
% \end{proof}

\section{Conclusion}
This paper is a first exploration of relational presheaves in concurrency theory from the point of view of their realizations.
Since the category of relational presheaves on a fixed category is well-behaved (finitely locally presentable),
it is a good setting to categorically and homotopically study some objects of interest, especially (higher dimensional) automata,
which will be the subject of a future work. 
In particular, relational automata seem to be a good alternative to automata with $\epsilon$-transitions~\cite[\S 1.1.4]{sakarovitch2009elements}. This could be for instance exploited in proving Kleene-like theorems (for example \cite{fahrenberg2024kleene}).
The simple combinatorial expression of the blowup also suggests to study directly this construction in relational precubical sets,
and a homotopical (in the sense of Quillen) definition of euclidean relational precubical sets seems plausible.
Finally, on an abstract level, it would be interesting to develop furthur the theory of relational presheaves,
for example by adding a Grothendieck topology on the base category and trying to define relational \emph{sheaves}.

% \newpage
\bibliographystyle{splncs04}
\bibliography{b}

\newpage
\appendix
\crefalias{section}{appendix}

\section{Variants of relational presheaves}
\label{sec:variants}
We write $\Par$ for the category of sets and partial functions.
There is a canonical functor $\Par\to\Rel$ sending a partial function to its graph, allowing us to consider~$\Par$ as a subcategory of~$\Rel$.

A \emph{quiver map} $F:\ccat\to \dcat$ between two categories $\ccat$ and $\dcat$ is a morphism of directed graphs between the underlying graphs of $\ccat$ and $\dcat$.

\begin{definition}
  We introduce the following variants of relational presheaves.
  \begin{enumerate}[(i)]
  \item A \emph{relational family} $P$ over $\ccat$ is a quiver map $\ccat^\op\to \Rel$. 
    We use the same notation as for relational presheaves.
  \item A \emph{partial presheaf} $P$ over $\ccat$ is a lax functor $\ccat^\op\to \Par$, 
    i.e. a collection of sets $(P(c))_{c\in\Ob(\ccat)}$ and partial functions $(P(f))_{f\in\Mor(\ccat)}$, 
    subject to the following:
    \[P(\id_c)=\id_{P(c)},\quad P(f)\circ P(g)= P(f\circ g) \text{ when both are defined}\]
  \end{enumerate}
\end{definition}
Obviously every relational presheaf is a relational family. 
Note that a partial presheaf $P$ is exactly a relational presheaves such that $P(f)$ is a partial function for all morphisms $f$ of $\ccat$.
%Now we define what is a \emph{morphism} of relational presheaves.
%Many definitions exist in the litterature (see \cite{sobocinski2015relational} for a discussion), 
%and we chose the one which is the most natural ($2$-)categorically,
%and which is also the best one from the model-theoretic point of view,
%since it will correspond to homomorphisms of models.

\begin{definition}
  We have the following categories associated to the previous variants.
  \begin{enumerate}[(i)]
  \item A \emph{morphism} of relational families $\alpha:P\Rightarrow  F$ is an oplax natural transformation, 
    i.e. a collection of functions $(P(c)\xrightarrow{\alpha_c}F(c))_{c\in\Ob(\ccat)}$ such that 
    $a\to _f b$ implies $\alpha_c(a)\to _f \alpha_d(b)$ for every $f:d\to c$.
  \item The category of relational families and morphisms of relational families over $\ccat$ is noted $\RelFam(\ccat)$. 
  \item The full subcategory of $\RelPsh(\ccat)$ consisting of partial presheaves in noted $\ParPsh(\ccat)$.
  \end{enumerate}
\end{definition}

\begin{definition}\label{def:Tcrel}
  We define the theory $\Tcpar$ on $\Sigma_\ccat^\Rel$ as an extension of $\Tcrel$ with the following axiom:
    \[R_f(x,z)\AND R_f(y,z) \vdash_{x:c,y:c,z:d} x=y\]
    for every morphism $f:d\to c$.
\end{definition}
The following is straightforward:

\begin{proposition}
  $\RelFam(\ccat)$ is the category of set-based $\Sigma_\ccat^\Rel$-structures, and $\ParPsh(\ccat)$ is the category of set-based models of $\Tcpar$.
\end{proposition}

For the realization, everything essentially works the same with relational families or partial presheaves instead of relational presheaves. 
A model corresponding to a relational family does not have to satisfy conditions~\ref{Tmod-refl} and~\ref{Tmod-trans} of \cref{Tmod},
and a model $M$ corresponding to a partial presheaf has to satisfy an additional condition: the obvious functor 
$M(c)\sqcup M(c)\sqcup M(d)\to M(R_f)\sqcup_{M(d)}M(R_f)$ factors through the codiagonal $M(c)\sqcup M(c)\sqcup M(d)\to M(c)\sqcup M(d)$.
We also have a version of \cref{prop:kan}, with essentially the same proof:

\begin{proposition}\label{prop:kan-pre}
	For every $\Sigma_\ccat^\Rel$-structure $M:\ccat_\Rel\to\dcat^\op$ in $\dcat^\op$, the corresponding realization $\hat{M}:\RelFam(\ccat)\to\dcat$ is given by:
  \[\hat{M}(P)=\colim_{(x,a)\in \int F(P)}M(x)\]
\end{proposition}
We end this section by a useful lemma:
\begin{lemma}\label{lem:closure}%[closure under composition]
  The inclusion functor $U':\RelPsh(\ccat)\hookrightarrow\RelFam(\ccat)$ admits a left adjoint.
\end{lemma}
\begin{proof}
	Let $M$ be a $\Sigma_\ccat^\Rel$-structure, $N$ a model of $\Tcrel$. 
	We build a model $M_\mathrm{comp}$ of $\Tcrel$ by induction:
	\begin{enumerate}
		\item $M_0$ is the structure having same values as $M$ on sorts and non-identity morphisms, and with $M_0(\id_c):=\Delta_{M(c)}\sqcup M(\id_c)$ for every object $c$;
		\item given $M_n$, $M_{n+1}$ is the structure whose value on sorts is the same as $M_n$, and such that 
		\[M_{n+1}(R_f):=\bigcup_{f=g\circ h}M_n(R_g)\times_{M_n(\dom(g))}M_n(R_h)\]
		(as subobjects of $M_n(\dom(f))\times M_n(\cod(f))$);
		\item $M_\mathrm{comp}=\colim_{n\in\N}M_n$.
	\end{enumerate}
	To see that $M_\mathrm{comp}$ is a model of $\Tcrel$,
	consider a situation $f=g\circ h$ in $\ccat$, $x\to _g y$ and $y\to _h z$ in $M_\mathrm{comp}$.
	Then there is some $n\in\N$ such that $x\to _g y$ and $y\to _h z$ in $M_n$.
	But then $x\to _{g\circ h} z$ in $M_{n+1}$ by definition, so it is true in $M_\mathrm{comp}$.
	Now we prove by induction that $\Hom_{\Sigma_\ccat^\Rel}(M,N)=\Hom_{\Sigma_\ccat^\Rel}(M_n,N)$ for all $n\in\N$. 
	It is true for $n=0$. Suppose it is true for some $n$.
	Take $\alpha:M_n\to N$, $f=g\circ h$, $x\to _g y$, $y\to _h z$ in $M_n$.
	Then $\alpha(x)\to _g \alpha(y)$ and $\alpha(y)\to _h\alpha(z)$,
	but then $\alpha(x)\to _{g\circ h}\alpha(z)$ because $N$ is a model of $\Tcrel$.
	This concludes the induction.
	Now by a similar argument, functoriality of this construction is clear (the action on morphisms is the identity). 
	This finishes the proof because
	\begin{align*}
			\Hom_{\Sigma_\ccat^\Rel}(M_\mathrm{comp},N) & =  \lim_{n\in\N}\Hom_{\Sigma_\ccat^\Rel}(M_n,N)\\
			                                            & = \lim_{n\in\N}\Hom_{\Sigma_\ccat^\Rel}(M,N)\\
														& = \Hom_{\Sigma_\ccat^\Rel}(M,N)
	\end{align*}
        \qed
\end{proof}
With the notation of the previous proof, $M_\mathrm{comp}$ will be called the \emph{closure under composition} of $M$.
Note that the inclusion is full and faithful, so the counit of the adjunction is an isomorphism by Theorem IV.3.1 of \cite{mac1998categories},
i.e.~if we close a relational presheaf under composition, we get the same relational presheaf.
This is also obvious by construction.

\section{Presheaves as relational presheaves}\label{sec:presheaves}
First we give a proof of \cref{th:U}
\begin{proposition}
  $U$ is full and faithful.
\end{proposition}
\begin{proof}
  Faithfulness follows immediately from the fact that the inclusion functor $\Set\to\Rel$ is faithful.
  Consider a morphism $\alpha:U(P)\to U(Q)$: it consists in a family of functions $\alpha_c:P(c)\to Q(c)$ indexed by objects $c\in\ccat$ which preserve the relations.
  So for all $f:d\to c$ in $\ccat$, $x\in P(c)$, $y\in P(d)$, we have
  \begin{align*}
    y=x.f & \Leftrightarrow (x,y)\in U(P)(f) \\
    & \Rightarrow (\alpha_d(x),\alpha_c(y))\in U(Q)(f)\\
    & \Leftrightarrow \alpha_c(y)=\alpha_d(x).f
  \end{align*}
  where $x.f$ is a notation of $P(f)(x)$.
  Therefore, $\alpha_c(x.f)=\alpha_c(x).f$ for all $f$ and~$x$, which proves fullness.
  \qed
\end{proof}

\begin{proposition}\label{prop:radj}
	$U$ has a right adjoint.
\end{proposition}
\begin{proof}
  By the special adjoint functor theorem (\cite[Corollary \S V.8]{mac1998categories} and \cite[Theorem 1.58]{adamek1994locally}),
  it is enough to show that $U$ preserves colimits, which we do by computing explicitly colimits in relational presheaves and observing that they coincide with ones in presheaves. Since coproducts and coequalizers generate all colimits, it is enough to handle those two particular cases.

  For coproducts, consider a family $(P_i)_{i\in I}$ of presheaves indexed by a set $I$. 
  We define a relational presheaf $P$ by setting $P(c)=\bigsqcup_{i\in I} P_i(c)$ for every $c\in\ccat$ and
  \[P(f)=\{(x,y)\,|\,\exists i\in I,\, x\in P_i(c),\, y\in P_i(d),\, (x,y)\in P_i(f)\}\]
  for all $f:c\to d$. 
  It is straightforward to check that this is indeed the coproduct in $\RelPsh(\ccat)$, 
  and that it coincides with the usual coprroduct in $\Psh(\ccat)$.

  For coequalizers, consider a diagram $P\rightrightarrows Q$, with morphisms $\alpha$ and $\beta$. 
  Now define a relational presheaf $R$ by setting 
  $R(c)$ to be the quotient of $Q(c)$ be the equivalence relation generated by $\alpha(x)\sim\beta(x)$ for all $x\in P(c)$, for all $c$,
  and $[x]\to_f[y]$ if there are $x'\sim x$ and $y'\sim y$ such that $x'\to_f y'$.
  However, this does not necessarily define a relational presheaf, so we need moreover to close relations under composition (see \cref{lem:closure}).
  Now it is easy to check that $R$ is the coequalizer of the diagram, 
  and that it is the usual coequalizer for presheaves 
  (we do not need to close under composition in this case).
  \qed
\end{proof}

\begin{remark}\label{rem:psh-to-pre}
  Since we do not need to close under composition for presheaves in the previous proof,
the inclusion $\Psh(\ccat)\to \RelFam(\ccat)$ also preserves colimits,
since colimits in $\RelFam(\ccat)$ are computed the same way, but without closing by composition.
\end{remark}

\begin{proposition}
	$U$ has a left adjoint.
\end{proposition}
\begin{proof}
	By the special adjoint functor theorem, $U$ has a left adjoint if and only if it is acessible and preserves small limits~\cite[Theorem 1.66]{adamek1994locally}.
	By \cref{prop:radj}, $U$ is accessible.
	By Corollary 3.2 of \cite{niefield2004change}, limits in $\RelPsh(\ccat)$ are computed pointwise, 
	which implies that $U$ preserves them,
	since limits in $\Psh(\ccat)$ are also computed pointwise.
        \qed
\end{proof}

\begin{remark}\label{rem:explicit}
In fact, we can give an explicit construction of the left adjoint.
Let $P$ be a relational presheaf.
First set \[P_0(c):=\{(x,f)\,|\, f:c\to \cod(f),\, x\in P(\cod(f))\}\] for all $c$,
\begin{align*}
		P_0(f) := & \{((x,\id_c),(y,\id_d))\,|\, (x,y)\in P(f)\}\\
		            & \cup \{((x,g),(x,g\circ f))\,|\, g:c\to d,\, x\in P(d)\}
\end{align*}
for all $f:b\to c$. 
In the example of relational precubical sets, $P_0$ is the result of freely adjoining faces to the existing cubes.
Now let $P_1$ be the closure of $P_0$ under composition.
Finally, let $P_2$ be defined in the following way: $P_2(c)$ is the quotient of $P_1(c)$ by the equivalence relation generated by: 
$x\sim y$ if and only if there is $f:c\to d$ and $z\in P_1(d)$ such that $z\to _f x$ and $z\to _f y$.
Two equivalence classes blong to a relation if there are representatives of each which belong to this relation.
It is not hard to check that this indeed is a relational presheaf whose relations are functional,
and that it defines a left adjoint to $U$ 
(essentially because free constructions, quotients and closure under composition are all left adjoints).
Let us just explain why we do not need to close by composition again.
Suppose given $[x]\to_f [y]$ and $[y]\to_g [z]$ in $P_2$.
We can suppose that $x\to_f y'$, $y\to z$ and there is $e$ and $h$ such that $e\to_h y$ and $e\to_h y'$.
Note that by construction, there is $z'$ such that $y'\to_g z'$. 
But since $P_1$ is already closed under composition, we have $e\to_{g\circ h} z$ and $e\to_{g\circ h} z'$,
or in other words $z'\in [z]$.
Now $x\to_{g\circ f}z'$, so $[x]\to_{g\circ f}[z]$.
\end{remark}
For example, applying respectively the left and right adjoints on the central relational graph below gives the left and right graphs respectively.
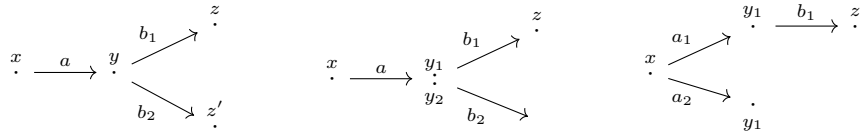
\[
  \begin{tikzcd}[row sep=0pt]
    &&\overset{z}\cdot\\
    \overset x\cdot\ar[r,"a"]&\overset y\cdot\ar[ur,"b_1"]\ar[dr,"b_2"']\\
    &&\overset{z'}\cdot
  \end{tikzcd}
  \qquad\quad
  \begin{tikzcd}[row sep=0pt]
    &&\overset{z}\cdot\\
    \overset x\cdot\ar[r,"a"]&\overset{y_1}{\underset{y_2}:}\ar[ur,"b_1"]\ar[dr,"b_2"']\\
    && {}
  \end{tikzcd}
  \qquad\quad
  \begin{tikzcd}[row sep=0pt]
    & {\overset {y_1}{\cdot} } & \overset z\cdot \\
    {\overset x\cdot} \\
    & {\underset {y_1}{\cdot} }
    \arrow["{b_1}", from=1-2, to=1-3]
    \arrow["{a_1}", from=2-1, to=1-2]
    \arrow["{a_2}"', from=2-1, to=3-2]
  \end{tikzcd}
\]

We can also compare these categories to the category of partial presheaves. 
The inclusion $U:\Psh(\ccat)\to \RelPsh(\ccat)$ factors through the category of partial presheaves $\ParPsh(\ccat)$,
so the inclusion $\Psh(\ccat)\to \ParPsh(\ccat)$ has right and left adjoints,
which are the restrictions of the ones defined above.
We can also prove that the inclusion $\ParPsh(\ccat)\to \RelPsh(\ccat)$ has a left adjoint.
Indeed, it is accessible by functoriality of the point construction (see~\cite[\S 2.1]{henry2019abstract}), 
and it is easy to see that it preserves limits, 
i.e.~that the additional axiom is still valid when we take the limit in relational presheaves,
because this axiom can itself be expressed using limits,
and limits are computed pointwise.
Concretely, it is similar to the left adjoint defined in Remark~\ref{rem:explicit}, without the first step where we freely add faces.
However, it does not have a right adjoint, since it does not preserve pushouts (e.g. the pushout square of \cite[\S 4.1]{dubut2019trees}).
This means that a realization of relational presheaves does not induce a realization of partial presheaves.

\section{Topology proofs}\label{sec:top}
\begin{proposition}
  Let $X$ be the pushout in the following cocartesian square:
  \vspace{2ex}
  \begin{center}
    \begin{tikzpicture}
    \draw[thick] (0.5,0) -- (0.5,1);

    \begin{scope}[xshift=15mm]
      \draw (0.5,0.5) node{$\xrightarrow{}$};
    \end{scope}

    \begin{scope}[xshift=30mm]
      \draw[thick] (0.5,0) -- (0.5,1);
      \filldraw (0.5,0) circle (.025);
    \end{scope}

    \begin{scope}[yshift=-15mm]
      \draw (0.5,0.5) node{$\downarrow$};
    \end{scope}

    \begin{scope}[yshift=-30mm]
      \fill[color=mygray] (0,0) rectangle (1,1);
      \draw[thick] (0,0) -- (0,1);
    \end{scope}

    \begin{scope}[xshift=15mm, yshift=-30mm]
      \draw (0.5,0.5) node{$\xrightarrow{}$};
      \draw (0.5,0.2) node{\small $\iota_1$};
    \end{scope}

    \begin{scope}[xshift=30mm, yshift=-15mm]
      \draw (0.5,0.5) node{$\downarrow$};
      \draw (0.8,0.5) node{\small $\iota_2$};
    \end{scope}

    \begin{scope}[xshift=30mm, yshift=-30mm]
      \fill[color=mygray] (0,0) rectangle (1,1);
      \draw[thick] (0,0) -- (0,1);
      \filldraw (0,0) circle (.025);
      \draw (-0.2,-0.1) node{$s$};
      \draw (0.5,0.5) node{$c$};
    \end{scope}
  \end{tikzpicture}
  \end{center}
  (where $c$ is the open square).
  For any connected compact space $K$, for any continuous map $f:K\to X$ such that $f(K)\subseteq c\cup\{s\}$,
  either $f$ is the contant map at $s$, or $f(K)\subseteq c$.
\end{proposition}
\begin{proof}
  The first step is to prove that $f^{-1}(\{s\})=f^{-1}(X\setminus f(K\setminus f^{-1}(\{s\})))$.
  \begin{align*}
    x\in f^{-1}(X\setminus f(K\setminus f^{-1}(\{s\}))) & \implies f(x) \not\in f(K\setminus f^{-1}(\{s\}))\\
                                                       & \implies x\not\in K\setminus f^{-1}(\{s\})\\
                                                       & \implies x\in f^{-1}(\{s\})
  \end{align*}
  Conversely, just notice that $s\not\in f(K\setminus f^{-1}(\{s\}))$. 
  Now we prove that $f^{-1}(\{s\})$ is open.
  By the first step, it is enough to prove that $Z:=X\setminus f(K\setminus f^{-1}(\{s\}))$ is open.
  The topology on $X$ is the pushout topology, so we just need to prove that 
  $\iota_i^{-1}(Z)$ is open for $i=1,\,2$. But $f(K\setminus f^{-1}(\{s\}))\subseteq c$ so $\iota_2^{-1}(Z)$ is the whole space, thus open, and 
  \[\iota_1^{-1}(Z)=\iota_1^{-1}(X\setminus f(K))\]
  which is open because $f(K)$ is closed, because $K$ is compact and $f$ is continuous so it sends compact spaces to compact spaces, and compact spaces are closed.
  This is enough to conclude, since now $K$ is covered by two disjoint open sets $f^{-1}\{s\}$ and $f^{-1}(c)$, and $K$ is connected so one of them is empty.
        \qed
\end{proof}

\begin{corollary}\label{cor:ht}
  The subspace $c\cup\{s\}\subseteq X$ of the previous proposition has the same homotopy type as the coproduct space $c\sqcup \{s\}$.
\end{corollary}
\begin{proof}
  Homotopy types are entirely determined by continuous functions whose domain is $[0,1]^n$ for every natural number $n$, which are compact connected spaces.
        \qed
\end{proof}

%\fill[color=mygray] (0,0) rectangle (1,1);
%  \filldraw (0,0) circle (.025);
%
% \begin{scope}[xshift=15mm]
%  \draw (0.5,0.5) node{$\hookrightarrow$};
%  \end{scope}

%  \begin{scope}[xshift=30mm]
%  \fill[color=mygray] (0,0) rectangle (1,1);
%  \draw[thick] (0,0) -- (0,1);
%  \filldraw (0,0) circle (.025);
%  \draw (-0.2,-0.1) node{$s$};
%  \end{scope}
Now we want to prove Proposition~\ref{prop:induced}. We have to go through a concrete description of the geometric realization, at leat for a subclass of relational precubical sets.

\begin{definition}
	Let $P$ be a relational presheaf, $c$ a cube of $P$.
	The \emph{combinatorial positive neighborhood} of $c$ is given by
	\[N^+(c):=\{(a,f)\,|\,a\to _f c\}\]
\end{definition}
Recall that a morphism $f:n\to n+k$ in~$\square$ decomposes uniquely as
\[
  f=d^{\epsilon_k}_{n+k-1,i_k}\circ\ldots \circ d^{\epsilon_2}_{n+1,i_2}\circ d^{\epsilon_1}_{n,i_1}
\]
with $0\leq i_1<i_2<\ldots<i_k<n+k$, so it can be seen as a word $w\in\{-1,0,1\}^{n+k}$ with exactly $k$ non-zero components, 
where $w_{i_j}=\epsilon_j$ and $w_i=0$ for the other indices (where $w_i$ denotes the $i$th letter of $w$).
We use this notation for the remaining of this section.
Using Proposition~\ref{prop:kan-pre}, we see that the underlying set of the geometric realization of a relational precubical set $P$ is $\bigsqcup_{n\in\N}P(n)\times]0,1[^n$.

\begin{lemma}\label{lem:basis}
	Let $P$ be a relational precubical set, $c$ an $n$-cube of $P$, $n\in\N$,
	$x$ a point of $|P|_\Top$ belonging to $c$, i.e.~to $\{c\}\times\mathopen]0,1[^n$. 
	Suppose that $N^+(c)$ is finite.
	Then a basis of neighborhoods of $x$ is given by:
	\[U_k(x):=\bigcup_{(a,w)\in N^+(c)}\{a\}\times\prod_{1\leq i \leq \dim(a)}I_{k,w_i,i}(x)\]
	where we define $I_{k,-1,i}(x):=\mathopen]0,2/k[$, $I_{k,0,i}(x):=\mathopen]x_{p(i)}-1/k,x_{p(i)}+1/k[$ and $I_{k,1,i}(x):=\mathopen]1-2/k,1[$
	with $p(i):=|\{j\in\llbracket 1,i\rrbracket\,|\,w_j=0\}|$, for every $i\in\llbracket 1,\dim(a)\rrbracket$ and every big enough $k>0$.
\end{lemma}
\begin{proof}
	From the definition of the realization as an inclusion followed by a Kan extension (\cref{prop:kan-pre}), 
	we immediately get that an open set of $|P|_\Top$ is a subset which is open in every $\{a\}\times]0,1[^{\dim(a)}$ and every
	\[\{a\}\times\mathopen]0,1\mathclose[^{\dim(a)}\cup\{a'\}\times\mathopen]0,1\mathclose[^{\dim(a')}\subseteq [0,1]^{\dim(a)}\] 
	with the induced topology, for $a\to _w a'$ (for any $a$, $a'$, $w$), where the inclusion sends $\{a'\}\times\oint{0,1}^{\dim(a')}$ to the corresponding face with respect to $w$ (see the beginning of section \ref{seq-real}).
	The remaining of the proof is topological routine, 
	using the finiteness of $N^+(c)$ to find a big enough $k$.
        \qed
\end{proof}
For example, in the closed $2$-cube, we represent an open neighborhood (in dark gray) for the points $x,\, y$ and $z$ respectively:
\[
  \begin{tikzpicture}[baseline=(b.base)]
      \coordinate (b) at (.5,.5);
      \fill[mygray] (0,0) rectangle (2,2);
      \fill[gray] (0.75,0) rectangle (1.25,0.5);
      \filldraw (0,0) circle (0.05);
      \filldraw (2,0) circle (0.05);
      \filldraw (0,2) circle (0.05);
      \filldraw (2,2) circle (0.05);
      \filldraw (1,0) circle (0.05) node[below]{$x$};
      \draw
      (0,0) edge[below,thick] (2,0)
      (2,0) edge[below,thick] (2,2)
      (0,0) edge[below,thick] (0,2)
      (0,2) edge[below,thick] (2,2)
      ;
    \end{tikzpicture}
  \qquad\qquad
  \begin{tikzpicture}[baseline=(b.base)]
      \coordinate (b) at (.5,.5);
      \fill[mygray] (0,0) rectangle (2,2);
      \fill[gray] (1.5,0) rectangle (2,0.5);
      \filldraw (0,0) circle (0.05);
      \filldraw (2,0) circle (0.05) node[below right]{$y$};
      \filldraw (0,2) circle (0.05);
      \filldraw (2,2) circle (0.05);
      \draw
      (0,0) edge[below,thick] (2,0)
      (2,0) edge[below,thick] (2,2)
      (0,0) edge[below,thick] (0,2)
      (0,2) edge[below,thick] (2,2)
      ;
    \end{tikzpicture}
    \qquad\quad
  \begin{tikzpicture}[baseline=(b.base)]
      \coordinate (b) at (.5,.5);
      \fill[mygray] (0,0) rectangle (2,2);
      \fill[gray] (0.75,0.75) rectangle (1.25,1.25);
      \filldraw (0,0) circle (0.05);
      \filldraw (2,0) circle (0.05);
      \filldraw (0,2) circle (0.05);
      \filldraw (2,2) circle (0.05);
      \filldraw (1,1) circle (0.03) node[below right]{$z$};
      \draw
      (0,0) edge[below,thick] (2,0)
      (2,0) edge[below,thick] (2,2)
      (0,0) edge[below,thick] (0,2)
      (0,2) edge[below,thick] (2,2)
      ;
    \end{tikzpicture}
\]

\noindent
As a corollary, we finally have

\inducedrestate*
\begin{proof}
  This is a cocontinuous functor, so we just need to show that it coincides with the geometric realization on $\square$.
  But this is straightforward from \cref{lem:basis}.
  \qed
\end{proof}

\begin{remark}
  In fact we can prove that this geometric realization coincides with~\cite{dubut2019trees} on partial presheaves which are ``sufficiently connected''.
  More generally, we can show that the geometric realization of relational precubical families preserves equalizers.
  Indeed, recall that a monomorphism is \emph{regular} if it is the equalizer of some pair of morphisms.
  So in $\RelFam(\square)$, if a partial precubical set $P$ is a regular subobject of its ``completion'' $P'$ as defined in~\cite{dubut2019trees} (which is just the result of applying the left adjoint of \cref{sec:presheaves}),
  then this regular monomorphism $P\hookrightarrow P'$ is sent to a regular monomorphism in $\Top$,
  in other words $|P|_\Top$ is a subspace of $|P'|_\Top$ (which is, as we just saw, the standard geometric realization of the precubical set~$P'$),
  which is exactly how the geometric realization is defined in~\cite{dubut2019trees}.
  Now the proof of the fact that $|{-}|_\Top$ preserves equalizers again just follows from the concrete description of the realization given by \cref{prop:kan-pre}.
  Note that $|{-}|_\Top$ does not send all monomorphisms to subspace inclusions, for example just consider the monomorphism:
  \[ \begin{tikzpicture}[baseline={([yshift=-.5ex]current bounding box.center)}]
    \filldraw (0,0) circle (.05);
  \end{tikzpicture}
  \sqcup
  \begin{tikzpicture}[baseline={([yshift=-.5ex]current bounding box.center)}]
    \draw[thick] (0,0) -- (.5,0);
  \end{tikzpicture}
  \quad\hookrightarrow\quad
  \begin{tikzpicture}[baseline={([yshift=-.5ex]current bounding box.center)}]
    \filldraw (0,0) circle (.05);
    \draw[thick] (0,0) -- (.5,0);
  \end{tikzpicture}
  \]
  The key property of regular monomorphisms $i:P\hookrightarrow P'$ that is implicitely used here is that if $x\to_f y$ in $P'$ and $x$ and $y$ are in $P$ then $x\to_f y$ in $P$;
  in other words, $i$ is an embedding, in the model-theoretic sense.
\end{remark}

\section{Combinatorial blowup}
\label{sec:comb-bup}
There is a classical equivalence between the category~$\Psh(\ccat)$ of presheaves over a category~$\ccat$ and discrete fibrations, see \cite[Theorem 2.1.2]{loregian2020categorical} for instance. This generalizes to the relational setting as follows.

\begin{definition}
  Let $P$ be a relational presheaf over a category $\ccat$.
  A morphism $\alpha:\tilde{P}\to P$ over $P$ is a \emph{discrete fibration} if for every morphism $f:d\to c$ of~$\ccat$, 
  for every $x\to _f y$ of $P$,
  for every $y'\in\alpha_d^{-1}(y)$,
  there is a unique $x'\in\alpha_c^{-1}(x)$ such that $x'\to _f y'$ in $\tilde{P}$.
  The full subcategory of $\RelPsh(\ccat)/P$ whose objects are the fibrations over $P$ is noted $\DFib(P)$.
\end{definition}

\begin{definition}
  Let $P$ be a relational presheaf over a category $\ccat$. 
  The \emph{category of elements} of $P$, noted $\int P$, is defined as the category whose
  \begin{itemize}
  \item objects are the couples $(c,x)$ with $c$ an object of $\ccat$ and $x \in P(c)$,
  \item morphisms $(d,y)\to (c,x)$ are the morphisms $f:d\to c$ such that $x\to _f y$.
  \end{itemize}
\end{definition}
\vspace{-2ex}%
\dfibrestate*
\begin{proof}
  First, we define the map $\psi:\DFib(P)\to \Psh((\int P)^\op)$. Given a fibration $\alpha:\tilde{P}\to P$, we define
  \[\psi(\alpha)(c,x):=\alpha_c^{-1}(x)\]
  for any object $(c,x)$ of $\int P$. Now given a morphism $f:(d,y)\to (c,x)$, 
  for an element $z\in \psi(\alpha)(d,y)$, there is a unique $z'\in\psi(\alpha)(c,x)$ such that $z'\to _f z$.
  We then set $\psi(\alpha)(w)(z)=z'$.
  This indeed defines a presheaf by the uniqueness of the lift.

  In the other direction, we define $\phi:\Psh((\int P)^\op)\to \DFib(P)$. 
  Take a presheaf $F$.
  We start by defining a relational presheaf $\tilde{P}$, in the following way:
  \begin{align*}
    \tilde{P}(c) & = \bigsqcup_{(c,x)\in\int P}F(c,x)\\
    \tilde{P}(f) & = \{(a,b)\,|\,\exists x,y,\, a\in F(c,x),\, b\in F(d,y),\, x\to _f y,\,F(w)(b)=a\}
  \end{align*}
  for every $c\in\ccat$, $f:d\to c$. 
  Now we can define $\phi(F):\tilde{P}\to P$ to be the projection $a\in F(c,x)\mapsto x$.
  It is a morphism of relational presheaves by construction, and checking that it is a fibration is a simple calculation.
	
  To finish the proof, we need to check that $\phi$ and $\psi$ are inverse to each other. 
  Again, this is a simple calculation.
  \qed
\end{proof}

Now we explain how to realize a presheaf $F\in \Psh((\int P)^\op)$.
Let $P$ be a relational precubical set.
As discussed in the proof of Lemma~\ref{lem:basis}, the underlying set of~$|P|_\Top$ is $\bigsqcup_{n\in\N}P(n)\times\oint{0,1}^n$,
and that an open set of~$|P|_\Top$ is a subset which is open in every $\set{a}\times\oint{0,1}^{\dim(e)}$ and every
\[\{a\}\times\oint{0,1}^{\dim(a)}\cup\{a'\}\times\oint{0,1}^{\dim(a')}\subseteq [0,1]^{\dim(a)}\] 
with the induced topology, for $a\to _f a'$ (for any $a$, $a'$, $f$).
We say that such an open set \emph{contains} a cube $c$ of $P$ if $\set{c}\times\oint{0,1}^{\dim(n)}\cap U\ne\varnothing$.
From this we deduce that there is a basis $\mathcal{B}$ of open sets of~$|P|_\Top$ such that every $U\in\mathcal{B}$ contains exactly one cube of minimal dimension.
We note $\min(U)$ the minimal cube of such $U$.
Now we can transform $F$ into a presheaf $F'$ over the poset category $\mathcal{B}$ by setting
\[F'(U):=F(\dim(\min(U)),\min(U))\]
For the action on morphisms, 
we have that if $U,\, V\in\mathcal{B}$ and $U\subseteq V$, 
then either $\min(V)\to_f\min(U)$ for some $f$, or $\min(U)=\min(V)$.
We can then define $F'(U\subseteq V):=F(f)$ in the first case, 
$F'(U\subseteq V):=\id$ in the second.
Now we can extend $F'$ to a presheaf on $|P|_\Top$ without changing its germs,
by right Kan extension along the inclusion $\mathcal{B}\to O(|P|_\Top)$ the poset of open sets of $|P|_\Top$,
\ie we get a presheaf 
\[F'':U\mapsto \lim_{\mathcal{B}\ni B\subseteq U}F'(U)\]
with obvious action on morphisms.
Finally, using the correspondance between sheaves on a space and local homeomorphisms over this space~\cite[\S II.6]{maclane2012sheaves},
we get a space $|F|$ corresponding to $F''$.
All this was already done in~\cite{chamoun2025non} to prove that the combinatorial blowup is indeed equivalent to the topological blowup.

\begin{definition}
  Let $P$ be a relational precubical set and $F\in \Psh((\int P)^\op)$.
  The geometric realization $|F|$ of $F$ over $|P|_\Top$ is the space defined by the above procedure.
\end{definition}
\begin{lemma}
  \label{lem:fib}
  Let $\alpha:P'\to P$ be a discrete fibration. 
  For every cube $c'$ of $P'$, there is a canonical bijection $N^+(c')\cong N^+(\alpha(c))$ induced by $\alpha$.
\end{lemma}
\begin{proof}
  By definition of a  discrete fibration.
  \qed
\end{proof}

\begin{lemma}
  \label{lem:fib-etale}
  Let $\alpha:P'\to P$ be a discrete fibration. Then $|\alpha|_\Top$ is a local homeomorphism.
\end{lemma}
\begin{proof}
  By \cref{lem:fib} and the fact that the realization locally only depends on~$N^+$.
\end{proof}

\rblowuprestate*
\begin{proof}
  With the notations of \cref{sec:blowup}, 
  for every $x\in|P|_\Top$, there is some cube $c$ of $P$ such that $x\in\{x\}\times\oint{0,1}^{\dim(c)}$.
  The germ of $F'$ at $x$ is given by $F''_x=F(\dim(c),c)$,
  because $\mathcal{B}$ forms a basis of the topology of $|P|_\Top$.
  So we have a set-theoretic bijection over $|P|_\Top$:
  \begin{align*}
      |F| =     & \bigsqcup_{x\in|P|_\Top}F''_x\\
          \cong & \bigsqcup_{(c,t)\in \bigsqcup_{n\in\N}P(n)\times\oint{0,1}^n}F(\dim(c),c)\\
          \cong & \bigsqcup_{c\in \bigsqcup_{n\in\N}P(n)}F(\dim(c),c)\times\oint{0,1}^{\dim(c)}\\
          =     & \bigsqcup_{c\in \bigsqcup_{n\in\N}P(n)}(\phi(F))^{-1}(c)\times\oint{0,1}^{\dim(c)}\\
          \cong & \bigsqcup_{n\in\N} \left(\bigsqcup_{c\in P(n)}(\phi(F))^{-1}(c)\right)\times\oint{0,1}^n\\
          \cong & \bigsqcup_{n\in\N} \phi(F)(n)\times\oint{0,1}^n\\
          =     & |\phi(F)|_\Top
  \end{align*}
  (we write $\phi(F)$ both for the morphism and its domain). Now it suffices to notice that $|F|\to|P|_\Top$ and $|\phi(F)|_\Top$ are both local homeomorphisms,
  by \cite[\S II.6]{maclane2012sheaves} and \cref{lem:fib-etale} respectively,
  so this bijection is in fact a local homeomorphism, thus a homeomorphism.
        \qed
\end{proof}
%
% We get the following description of the combinatorial blowup.

% \begin{definition}
  % Let $I$ be the relational precubical set consisting of two edges sharing a vertex
  % \[\to \cdot\to \]
  % For $n\in\N$, $I^{\otimes n}$ is the $n$-fold tensor product of $I$.
% \end{definition}

We still need to make the link between the combinatorial blowups of~\cite{chamoun2025non} and \cref{def:bup}.
Fix $n\in\N$, a relational precubical set $P$ of maximal dimension $n$, $k\leq n$, $c\in P(k)$.
Let $m(c)$ be the vertex $(c,(1/2,\ldots,1/2))$ of $\frac{1}{2}P$, \ie the vertex added at the center of the cube $c$.
Define $Z^0:=I$,\,$Z^1:=J$.
Note that in~\cite{chamoun2025non}, the combinatorial blowup of $P$ is expressed in terms of the precubical subsets of the $6$-subdivision of $P$ which are isomorphic to $I_{0,n}$. 
We have to subdivise that much essentially because of boundary issues, to ensure that we have isomorphisms.
Working with relational precubical subsets which are isomorphic to $N(I_{0,n})$ instead immediately simplifies the situation, and just requires to look at the barycentric subdivision of $P$.
Now we want an expression which does not use any subdivision at all,
so we need to replace isomorphisms by surjective local embeddings.
The following two results are enough, since it is easy to see that they are true ``up to permutation of the dimensions of the cubes''.

\begin{lemma}\label{lem1}
  For every monomorphism $i:N(I_{0,n})\to \frac{1}{2}P$ centered at $m(c)$ (\ie such that $i(\min(N(I_{0,n})))=m(c)$),
  there is a tuple $w\in\{0,1\}^n$ with exactly $k$ zeros, such that $i$ factors through the inclusion $N(I_{0,n})\hookrightarrow \frac{1}{2}N(\bigotimes_{i} Z^{w_i})$ sending $\min(N(I_{0,n}))$ to $\min(N(\bigotimes_{i} Z^{w_i}))$,
  via a morphism of the form $\frac{1}{2}f$ for $f:N(\bigotimes_{i} Z^{w_i})\to P$. 
  Moreover, $w$ and $f$ are unique.
\end{lemma}
\begin{proof}
  We have to prove that we can ``extend'' $i$ in exactly $2(n-k)$ dimensions. 
  There are exactly $2(n-k)$ edges in the image of $i$ which are of the form $(c,\mathbf{t})$ where $c$ is of dimension $k+1$.
  Indeed, all the $n$-cubes with are adjacent to $m(c)$ in $\frac{1}{2}P$ have exactly $n-k$ edges adjacent to $m(c)$ with are of this form, by definition of the subdivision;
  one edge is shared by $2^{n-1}$ $n$-cubes, and there are $2^n$ $n$-cubes, so the total number of such edges is $2(n-k)$.
  These edges represent the directions in which we can extend $i$.
  Now we just add all the cubes $(e,\mathbf{s})$ such that there is some cube $(e,\mathbf{s}')\in \Im(i)$.
  Details are easily checked.
\end{proof}

\begin{lemma}\label{lem2}
  Let $p:N(I_{k,n-k})\to P$. The induced morphism \[p':N(I_{0,n})\hookrightarrow \frac{1}{2}N(I_{k,n-k})\to \frac{1}{2}P\]
  is a monomorphism if and only if $p$ is a local embedding.
\end{lemma}
\begin{proof}
  Notice that $p$ is a local embedding if and only if $p$ is a local embedding \emph{at $\min(I_{k,n-k})$}, in the sense that $a\to_f\min(I_{k,n-k})$ and $b\to_f \min(I_{k,n-k})$ implies $p(a)\ne p(b)$.
  Indeed, on direction is straightforward. For the other one, suppose that there is some $e$ such that $a\to_f e$ and $b\to_f e$ for some $a,\,b,\,f$.
  By definition of $N(I_{k,n-k})$, there is $g$ such that $e\to_g\min(I_{k,n-k})$, so $a\to_{g\circ f}\min(I_{k,n-k})$ and $b\to_{g\circ f}\min(I_{k,n-k})$.

  Now suppose that $p$ is a local embedding and take two cubes $(a,\mathbf{t})$ and $(b,\mathbf{s})$ of $N(I_{0,n})\subseteq \frac{1}{2}N(I_{k,n-k})$ such that $p'((a,\mathbf{t}))=p'((b,\mathbf{s}))$. 
  Necessarily $p(a)=p(b)=:e$. But then $p'((a,\mathbf{t}))=(e,\mathbf{t})$ and $p'((b,\mathbf{s}))=(e,\mathbf{s})$, so $\mathbf{s}=\mathbf{t}$.
  But this means that $\min(I_{k,n-k})$ is in the same position relatively to $a$ and $b$, otherwise we would not have $(a,\mathbf{t}),(b,\mathbf{t})\in N(I_{0,n})$.
  In other words, there is $f$ such that $a\to_f \min(I_{k,n-k})$ and $b\to_f \min(I_{k,n-k})$.
  But $p$ is a local embedding, so $a=b$.
  The converse is similar, using the above characterisation of local embeddings.
      \qed
\end{proof}
\end{document}